\documentclass[a4paper,11pt]{amsart}
\usepackage{amsmath,amsthm,amssymb,amsfonts,enumerate,color,esint,bbm}
\usepackage[pdftex]{graphicx}
\usepackage{float}
\usepackage{tikz}
\usetikzlibrary{patterns}

\usepackage[shortlabels]{enumitem}
\usepackage{subfig}
\usepackage[colorlinks=true, allcolors=blue]{hyperref}
\usepackage{amsrefs}
\usepackage{pgfplots}
\pgfplotsset{
compat=newest,
colormap={blackwhite}{gray(0cm)=(.25); gray(1cm)=(1)}
}
\usetikzlibrary{hobby}
\usepackage{comment}

\newcommand{\N}{\mathbb{N}}
\newcommand{\R}{\mathbb{R}}
\newcommand{\Z}{\mathbb{Z}}

\newcommand{\abs}[1]{\left\vert#1\right\vert}
\def\({\left(}
\def\){\right)}

\newcommand{\one}{\mathbbm{1}}

\newcommand{\ep}{\varepsilon}

\newtheorem{thm}{Theorem}[section]
\newtheorem{prop}[thm]{Proposition}

\newtheorem{lem}[thm]{Lemma}

\theoremstyle{definition}
\newtheorem{defn}[thm]{Definition}
\newtheorem{rem}[thm]{Remark}

\numberwithin{equation}{section}

\allowdisplaybreaks

\author{Erisa Hasani}
\address{Department of Mathematics\\
The University of Texas at Austin\\
2515 Speedway, Austin\\
TX 78712, United States of America}
\email{ehasani@utexas.edu}

\author{Stefania Patrizi} 
\address{Department of Mathematics\\
The University of Texas at Austin\\
2515 Speedway, Austin\\
TX 78712, United States of America}
\email{spatrizi@math.utexas.edu}


\begin{document}

\title[Interacting fronts in the strongly nonlocal Allen--Cahn equation]{Interacting fronts in the strongly nonlocal Allen--Cahn equation}

\begin{abstract}
We study the sharp-interface limit of the fractional Allen--Cahn equation in $\R^n$, $n\geq 2$, in the strongly nonlocal regime $s\in(0,\frac12)$, for initial data consisting of finitely many nested transition layers. We identify a coupled geometric law for the motion of the resulting fronts: each velocity contains fractional mean curvature and a nonlocal interaction potential generated by the other fronts. These interactions persist while the fronts remain separated, in contrast to the independent motion in the critical regime $s=\frac12$. Assuming that the coupled law admits a smooth evolution with strictly nested sets on a given time interval, we prove that the Allen--Cahn solutions converge locally uniformly away from the moving fronts to the corresponding integer-valued phases. 
\end{abstract}

\maketitle

\section{Introduction}

We study the fractional Allen--Cahn equation
\begin{equation} \label{eq:pde}
\ep \partial_t u^{\ep} = \mathcal{I}^s_n [u^{\ep}]  -\frac{1}{\ep^{2s}}  W'(u^\ep) \quad \hbox{in}~(0,\infty)\times \R^n, ~n \geq 2,
\end{equation}
where $\ep>0$ is a small parameter, 
$\mathcal{I}^s_n=-c_{n,s}(-\Delta )^s$ denotes, up to a constant, the fractional Laplacian of order  $2s\in(0,1)$  in $\R^n$,  
and $W$ is a smooth multi-well periodic potential with wells on integers (see \eqref{eq:operator} and \eqref{eq:W} respectively). 
We focus on the singular regime $s\in(0,\frac12)$, corresponding to strongly nonlocal diffusion.  

The equation is related to the Peierls–Nabarro description of crystal dislocations \cite{PN1,PN2}; related one-dimensional and higher-dimensional formulations appear, for instance,  in \cite{CozziDavilaDelPino, DipierroFigalliValdinoci, DipierroPalatucciValdinoci, GonzalezMonneau, MeursPatrizi, MonneauPatrizi2,MonneauPatrizi,patval1, PatriziVaughan2}. We consider initial data formed by superposing transition layers around the boundaries of finitely many strictly nested sets. As $\varepsilon\to0$, these layers separate regions in which $u^\varepsilon$ approaches successive integer values. Our objective is to identify the motion of \emph{every} boundary in this limit.

For a single transition layer in the strongly nonlocal regime, the sharp-interface limit was established in \cite{HasaniPatrizi}, where the limiting interface moves by fractional mean curvature. At the critical exponent $s=\frac12$, under the scaling considered in \cite{PatriziVaughan2}, multiple fronts evolve independently by classical mean curvature, with no interaction term in their limiting velocities. For $s<\frac12$, by contrast, the contribution of one layer is felt at the locations of the others even while their boundaries remain a positive distance apart. The limiting law \eqref{eq:velocity-intro} makes this difference explicit: the velocity of each front contains both its fractional mean curvature and an interaction potential generated by the other fronts. The resulting evolution is a coupled system of geometric laws. Identifying this system, and proving that it governs the sharp-interface limit of \eqref{eq:pde}, is the principal result of this paper.

Theorem~\ref{thm:main_result} is conditional on the geometric evolution: we assume that \eqref{eq:velocity-intro} has a smooth solution on $[0,T]$ whose sets remain strictly nested. For a single front, Julin and La Manna \cite{JulinLaManna2020} proved short-time existence of smooth fractional mean-curvature flow from a bounded $C^{1,1}$ initial set. We do not establish general existence for the coupled system of multiple fronts.  Nevertheless, Theorem~\ref{thm:main_result} identifies the coupled motion of every front whenever a smooth, strictly nested evolution exists. Establishing short-time existence for \eqref{eq:velocity-intro} and developing a weak formulation that accounts for singularity formation are objectives of a future paper. Examples of singularity formation in nonlocal curvature flows \cite{CesaroniDipierroNovagaVal,CintiSinestrariValdinoci} show why a weak formulation is needed. Under the smooth-existence hypothesis, the solution of the fractional Allen--Cahn equation converges locally uniformly away from the moving boundaries to the sum of the characteristic functions of the nested sets. The hypothesis determines the time interval on which our convergence result applies. The parallel half-space and concentric-ball configurations in Appendix A give explicit classes of geometric evolutions and illustrate the interaction terms.

The stationary theory provides context for this result. In the strongly nonlocal regime, the sharp-interface limit of the corresponding stationary problem is governed by fractional perimeter \cite{SavinValdinoci}; related variational results appear in \cite{Alberti,AmbrosioDePhilippisMartinazzi,Contigarmul,GarroniMuller}. Nonlocal minimal surfaces and their fractional mean curvature were studied in \cite{CRS}. These results explain the curvature term for one interface. The additional integrals in \eqref{eq:velocity-intro} describe the dynamical effect of the other interfaces and are essential to the multiple-front problem.

\subsection{Setting of the problem and main result}
The operator $\mathcal{I}^s_n$ is a nonlocal integro-differential operator and is defined on functions $u \in C^{0,1}(\mathbb{R}^n)$ by
\begin{equation}\label{eq:operator}
\mathcal{I}_n^s u(x) 
 = \int_{\R^n} \left( u(x+z) - u(x)\right) \,\frac{dz}{\abs{z}^{n+2s}}, \quad x \in \R^n.
\end{equation}
For further background on fractional Laplacians, see for example \cites{Hitchhikers,Stinga}.

The potential $W$ satisfies
\begin{equation}\label{eq:W}
\begin{cases}
W \in C^{4, \beta} (\R) & \hbox{for some}~0 < \beta <1 \\
W(u+1) = W(u) & \hbox{for any}~ u \in \R\\
W= 0 & \hbox{on}~\Z\\
W>0 & \hbox{on}~\R \setminus \Z\\
W''(0) >0.
\end{cases}
\end{equation}

We let $u^\ep$ be the solution to \eqref{eq:pde} when the initial condition $u_0^\ep$ is a superposition of layer solutions.  The layer solution (also called the phase transition) $\phi :\R \to (0,1)$ is the unique solution to 
\begin{equation} \label{eq:standing wave}
\begin{cases}
 C_{n,s}  \mathcal{I}_1^s[\phi] = W'(\phi) & \hbox{in}~\R\\
 \dot{\phi}>0 & \hbox{in}~\R\\
\phi(-\infty) = 0, \quad \phi(+\infty)=1,\quad \phi(0) = \frac{1}{2},
\end{cases}
\end{equation}
where $\mathcal{I}_1^s$ denotes the nonlocal operator in \eqref{eq:operator} with $n=1$ and the constant $C_{n,s}>0$ (given explicitly in \eqref{eq:Cns}) depends only on $s \in (0,\frac{1}{2})$ and on the dimension $n$.

For a fixed $N \in \N$, let $(\Omega_0^i)_{i=1}^N$ be a finite sequence of subsets of $\R^n$ satisfying 
\begin{equation}\label{Omega_0^iassumptions}
\begin{cases}
 \Omega_0^{i}\text{ is an open  bounded domain with smooth boundary }\Gamma_0^i := \partial \Omega_0^i,\\
 \Omega_{0}^{i+1} \subset \subset \Omega_0^{i},\quad i=1,\ldots,N-1.
\end{cases}
\end{equation}
By the strict nesting condition, the boundaries $(\Gamma_0^i)_{i=1}^N$ are pairwise separated. 

Let $d_i^0(x)$ be the signed distance function associated to  $\Omega_0^i$, $i=1,\dots, N$, given by
\begin{equation}\label{eq:initial d_i}
d_i^0(x) = \begin{cases}
d(x, \Gamma_0^i) & \hbox{if}~x \in \Omega_0^i \\
-d(x,\Gamma_0^i) & \hbox{otherwise}.
\end{cases}
\end{equation}
We assume the following well-prepared initial condition for the solution of \eqref{eq:pde}
\begin{equation}\label{initial_data}
u_0^\ep(x) = \sum_{i=1}^N \phi\left(\frac{d_i^0(x)}{\ep}\right),
\end{equation}
see Figure 1. 

Let $(\Omega_t^i)_{i=1}^N$, with $t\in [0,T]$, be the smooth evolution of the initial family $(\Omega_0^i)_{i=1}^N$,  governed by the following system of laws. Denote $\Gamma_t^i:=\partial\Omega_t^i$.
 For each $i=1,\dots,N$, the set $\Omega_t^i$ evolves in the direction of its outward unit normal vector with scalar velocity
\begin{equation}\label{eq:velocity-intro}
v_i(t,x)=c_0\left(\frac{1}{2} H_{2s}(\Omega^i_t)(x)-\sum_{j=1}^{i-1} \int_{(\overline{\Omega_t^j})^c} \frac{dz}{|z-x|^{n+2s}}  +\sum_{j=i+1}^{N} \int_{\Omega^j_t} \frac{dz}{|z-x|^{n+2s}}\right),\qquad x\in \Gamma_t^i,
\end{equation}
where $ H_{2s}(\Omega^i_t)$ denotes the fractional mean curvature of order $2s$ of $\Omega^i_t$ and $c_0>0$ is an explicit constant (see \eqref{def:c0}).
See  Section \ref{sec:FMC} for the definition of the fractional mean curvature of a set. 
The second and third terms on the right-hand side of \eqref{eq:velocity-intro} describe the interaction  of $\Omega_t^i$ with the other evolving sets $\Omega_t^j$, $j\neq i$.
Due to the singularity at  $z=x$,  these interaction terms are well defined only as long as the evolving sets remain strictly nested. Accordingly, we require that 
\begin{equation}\label{nestingassumption}
\Omega_t^{i+1} \subset \subset \Omega_t^i
\qquad \text{for all } i=1,\ldots,N-1,\quad t\in[0,T].
\end{equation}

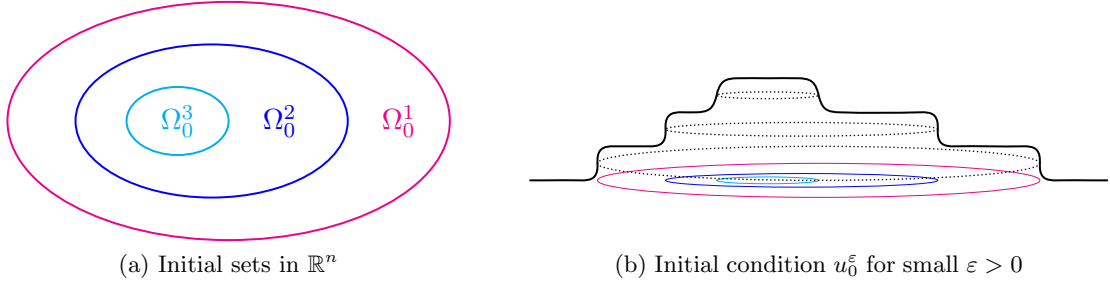
\begin{figure}[h]
\centering
\subfloat[Initial sets in $\R^n$]{
\begin{tikzpicture}[scale=0.45, use Hobby shortcut, closed=true]
\draw[magenta, line width=.75pt] (1.5,0) ellipse (6.5cm and 3.5cm);
\node[magenta] at (6.5,0) {$\Omega_0^1$};
\draw[blue, line width=.75pt] (1,0) ellipse (4cm and 2.25cm);
\node[blue] at (3,0) {$\Omega_0^2$};
\draw[cyan, line width=.75pt] (0,0) ellipse (1.5cm and 1cm);
\node[cyan] at (0,0) {$\Omega_0^3$};
\end{tikzpicture}
}
\qquad
\subfloat[Initial condition $u_0^{\ep}$ for small $\ep>0$]{
\begin{tikzpicture}[scale=0.45, use Hobby shortcut, closed=false]
\draw[magenta, line width=.25pt] (1.5,0) ellipse (6.5cm and .5cm);
\draw[blue, line width=.25pt] (1,0) ellipse (4cm and .2cm);
\draw[cyan, line width=.25pt] (0,0) ellipse (1.5cm and .1cm);
\draw[densely dotted, line width=.5pt] (1.5,.5) ellipse (6.5cm and .5cm);
\draw[densely dotted, line width=.5pt] (1,1.5) ellipse (4cm and .2cm);
\draw[densely dotted, line width=.5pt] (0,2.5) ellipse (1.5cm and .1cm);
\draw[line width=.75pt]
    (-7,0)..(-6.5,0)..(-5.5,0)..(-5.1,.1)..(-5,.5)..(-4.9,.9)..(-4.5,1)..
    (-3.5,1)..(-3.1,1.1)..(-3,1.5)..(-2.9,1.9)..(-2.5,2)..
    (-2.1,2)..(-1.7,2.1)..(-1.5,2.5)..(-1.3,2.9)..(-.9,3)..(0,3)..
    (.9,3)..(1.3,2.9)..(1.5,2.5)..(1.7,2.1)..(2.1,2)..
    (4.5,2)..(4.9,1.9)..(5,1.5)..(5.1,1.1)..(5.5,1)..
    (7.5,1)..(7.9,.9)..(8,.5)..(8.1,.1)..(8.5,0)..(9.5,0)..(10,0);
\node[opacity=0] at (1,-1.5) {};
\end{tikzpicture}
}
\caption{Initial configuration for $N=3$ in dimension $n=2$.}
\label{fig:initial}
\end{figure}

The following is our main result.

\begin{thm}\label{thm:main_result}
Assume that \eqref{eq:W} and \eqref{Omega_0^iassumptions} hold.  Let $u^\ep = u^\ep(t,x)$ be the unique solution of  \eqref{eq:pde} with initial datum $u_0^\ep: \R^n \to (0,N)$ defined by \eqref{initial_data}, 
where $\phi$ solves \eqref{eq:standing wave},  $d_i^0$ is given in \eqref{eq:initial d_i}, and $N\geq1$. Suppose that there exists $T>0$ such that the system of mean curvature flows \eqref{eq:velocity-intro} admits a smooth solution $(\Omega_t^i)_{i=1}^N$ on $[0,T]$ starting from $(\Omega_0^i)_{i=1}^N$ and satisfying \eqref{nestingassumption}. 
Then, as $\ep \to 0$, the solution $u^\ep$ satisfies
\[
u^\ep \to \begin{cases}
N & \Omega_t^N,\\
 i  \quad \,\, \hbox{locally uniformly in}~&\Omega_t^i\setminus\overline{\Omega_t^{i+1}}, \quad i=1,\dots, N-1,\\
 0 & \R^n\setminus\overline{\Omega_t^1}.
\end{cases}
\]
\end{thm}

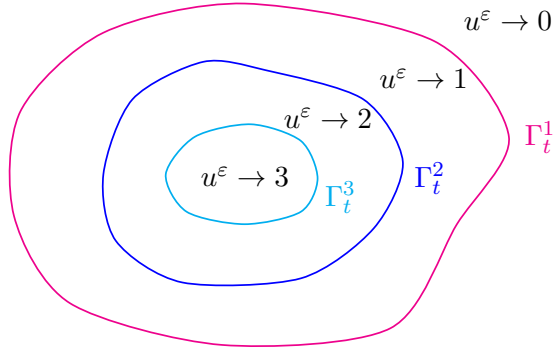
\begin{figure}[h]
\centering
\begin{tikzpicture}[scale=0.55, use Hobby shortcut, closed=true]

\def\Omegone{(0.0, 4.2) (4.55, 3.3) (6.35, 1.0) (5.1, -1.1)
            (3.5, -3.6) (-0.55, -4.0) (-4.2, -3.0) (-5.55, -0.9)
            (-5.4, 1.6) (-3.4, 3.5)};

\def\Omegtwo{(0.5, 2.6) (2.85, 1.9) (3.8, 0.4) (3.1, -1.2)
            (1.35, -2.4) (-1.5, -2.5) (-3.1, -1.5) (-3.4, 0.1)
            (-2.7, 1.9) (-1.0, 2.8)};

\def\Omegthree{(0.15, 1.3) (1.35, 0.9) (1.75, 0.0) (1.35, -0.8)
            (0.0, -1.1) (-1.4, -0.8) (-1.9, 0.1) (-1.2, 1.0)};

\draw[magenta, line width=.65pt] plot[smooth cycle] coordinates {\Omegone};
\draw[blue, line width=.65pt] plot[smooth cycle] coordinates {\Omegtwo};
\draw[cyan, line width=.65pt] plot[smooth cycle] coordinates {\Omegthree};

\node at (6.33, 3.8) {$u^{\ep} \to 0$};
\node at (4.3, 2.4) {$u^{\ep} \to 1$};
\node at (2.0, 1.45) {$u^{\ep} \to 2$};
\node at (0.0, 0.05) {$u^{\ep} \to 3$};

\node[magenta] at (7.1, 1.0) {$\Gamma_t^1$};
\node[blue] at (4.45, 0.0) {$\Gamma_t^2$};
\node[cyan] at (2.3, -0.45) {$\Gamma_t^3$};
\end{tikzpicture}
\caption{Convergence result in dimension $n=2$ for $N=3$ nested fronts.}
\label{fig:main thm}
\end{figure}

\subsection{Gradient flow formulation}\label{sec:grad_flow}

The coupled system \eqref{eq:velocity-intro}  is formally the $L^2$-gradient flow of a nonlocal partition energy, which we now describe. For bounded measurable sets $A,B\subset \R^n$, denote by
\[
\mathcal{L}(A,B):=\int_A\int_B \frac{dx\,dy}{|x-y|^{n+2s}}
\]
the pairwise nonlocal interaction between $A$ and $B$. The fractional $2s$-perimeter of $\Omega$ is $\operatorname{Per}_{2s}(\Omega):=\mathcal{L}(\Omega,\Omega^c)$, see \cite{CRS}.

For each $i\in\{1,\ldots,N\}$, we associate to the set  $\Omega^i$ the energy
\begin{equation}\label{eq:single_front_energy_intro}
\mathcal{J}(\Omega^i) := \mathcal{L}\big(\Omega^i,(\Omega^i)^c\big) \;+\!\sum_{j<i}\!\mathcal{L}\big(\Omega^i,(\Omega^j)^c\big) \;+\!\sum_{j>i}\!\mathcal{L}\big((\Omega^i)^c,\Omega^j\big).
\end{equation}
The first term is the fractional perimeter of $\Omega^i$. The second sum accounts for the interactions of $\Omega^i$ with the exteriors of the larger phases ($j<i$), and the third for the interactions of the exterior of   $\Omega^i$ with the smaller phases it contains ($j>i$). The contributions appearing in $\mathcal{J}(\Omega^i)$ are illustrated in Figure~\ref{fig:interactions} for $i=3$ and $N=4$.

The total energy of the partition is then given by
\begin{equation}\label{eq:total_energy_intro}
\mathcal{E}_N(\Omega^1,\ldots,\Omega^N) := \tfrac{1}{2}\sum_{i=1}^N \mathcal{J}(\Omega^i) = \tfrac{1}{2}\sum_{i=1}^N \mathcal{L}\big(\Omega^i,(\Omega^i)^c\big) \;+\!\sum_{1\leq j<i\leq N}\!\mathcal{L}\big(\Omega^i,(\Omega^j)^c\big).
\end{equation}

\begin{figure}[h]
\centering
\begin{tikzpicture}[scale=0.72,
    black hatch/.style={pattern=north east lines, pattern color=black!70},
    red hatch/.style={pattern=horizontal lines, pattern color=red!70}]

\def\drawgammaone{plot[smooth cycle, tension=0.7] coordinates {
    (-3.1,0.3) (-2.5,2.0) (-0.8,2.9) (1.2,2.7) (2.8,1.8) (3.2,0.2) (2.7,-1.6) (1.0,-2.6) (-1.2,-2.5) (-2.6,-1.5)
}}
\def\drawgammatwo{plot[smooth cycle, tension=0.7] coordinates {
    (-2.1,0.2) (-1.7,1.4) (-0.4,2.0) (0.9,1.8) (1.9,1.1) (2.2,0.0) (1.8,-1.1) (0.6,-1.7) (-0.7,-1.7) (-1.8,-0.9)
}}
\def\drawgammathree{plot[smooth cycle, tension=0.7] coordinates {
    (-1.2,0.15) (-0.9,0.9) (-0.1,1.2) (0.7,1.05) (1.1,0.5) (1.15,-0.15) (0.9,-0.7) (0.2,-1.0) (-0.5,-0.9) (-1.0,-0.4)
}}
\def\drawgammafour{plot[smooth cycle, tension=0.7] coordinates {
    (-0.4,0.1) (-0.3,0.45) (0.05,0.55) (0.35,0.4) (0.45,0.1) (0.35,-0.2) (0.05,-0.35) (-0.3,-0.2)
}}

\def\dx{9}
\def\dy{-8.5}

\begin{scope}[shift={(0,0)}]
\begin{scope}
    \clip (-4.2,-3.5) rectangle (4.2,3.8);
    \fill[black hatch] (-4.2,-3.5) rectangle (4.2,3.8);
    \fill[white] \drawgammathree;
\end{scope}
\fill[red hatch] \drawgammathree;
\draw[thick] \drawgammaone;
\draw[thick] \drawgammatwo;
\draw[thick, red] \drawgammathree;
\draw[thick] \drawgammafour;
\node at (2.0, -1.7) {$\Omega^1$};
\node at (1.2, -1.2) {$\Omega^2$};
\node[red] at (0.5, -0.6) {\small$\Omega^3$};
\node at (0.0, 0.1) {\small$\Omega^4$};
\node[anchor=north, font=\large] at (0, -3.6) {$\mathcal{L}\big(\Omega^3,\, (\Omega^3)^c\big)$};
\end{scope}

\begin{scope}[shift={(\dx,0)}]
\begin{scope}
    \clip (-4.2,-3.5) rectangle (4.2,3.8);
    \fill[red hatch] (-4.2,-3.5) rectangle (4.2,3.8);
    \fill[white] \drawgammathree;
\end{scope}
\fill[black hatch] \drawgammafour;
\draw[thick] \drawgammaone;
\draw[thick] \drawgammatwo;
\draw[thick, red] \drawgammathree;
\draw[thick] \drawgammafour;
\node at (2.0, -1.7) {$\Omega^1$};
\node at (1.2, -1.2) {$\Omega^2$};
\node[red] at (0.5, -0.6) {\small$\Omega^3$};
\node at (0.0, 0.1) {\small$\Omega^4$};
\node[anchor=north, font=\large] at (0, -3.6) {$\mathcal{L}\big((\Omega^3)^c,\, \Omega^4\big)$};
\end{scope}

\begin{scope}[shift={(0,\dy)}]
\begin{scope}
    \clip (-4.2,-3.5) rectangle (4.2,3.8);
    \fill[black hatch] (-4.2,-3.5) rectangle (4.2,3.8);
    \fill[white] \drawgammatwo;
\end{scope}
\fill[red hatch] \drawgammathree;
\draw[thick] \drawgammaone;
\draw[thick] \drawgammatwo;
\draw[thick, red] \drawgammathree;
\draw[thick] \drawgammafour;
\node at (2.0, -1.7) {$\Omega^1$};
\node at (1.2, -1.2) {$\Omega^2$};
\node[red] at (0.5, -0.6) {\small$\Omega^3$};
\node at (0.0, 0.1) {\small$\Omega^4$};
\node[anchor=north, font=\large] at (0, -3.6) {$\mathcal{L}\big(\Omega^3,\, (\Omega^2)^c\big)$};
\end{scope}

\begin{scope}[shift={(\dx,\dy)}]
\begin{scope}
    \clip (-4.2,-3.5) rectangle (4.2,3.8);
    \fill[black hatch] (-4.2,-3.5) rectangle (4.2,3.8);
    \fill[white] \drawgammaone;
\end{scope}
\fill[red hatch] \drawgammathree;
\draw[thick] \drawgammaone;
\draw[thick] \drawgammatwo;
\draw[thick, red] \drawgammathree;
\draw[thick] \drawgammafour;
\node at (2.0, -1.7) {$\Omega^1$};
\node at (1.2, -1.2) {$\Omega^2$};
\node[red] at (0.5, -0.6) {\small$\Omega^3$};
\node at (0.0, 0.1) {\small$\Omega^4$};
\node[anchor=north, font=\large] at (0, -3.6) {$\mathcal{L}\big(\Omega^3,\, (\Omega^1)^c\big)$};
\end{scope}

\end{tikzpicture}
\caption{The four contributions to $\mathcal{J}(\Omega^3)$ for $N=4$, with the front $\Gamma^3$ shown in red. Top left: the fractional perimeter $\mathcal{L}(\Omega^3,(\Omega^3)^c)=\operatorname{Per}_{2s}(\Omega^3)$. Top right: the interaction $\mathcal{L}((\Omega^3)^c,\Omega^4)$ with the smaller phase $\Omega^4$ contained in $\Omega^3$ ($j>i$). Bottom: the interactions $\mathcal{L}(\Omega^3,(\Omega^2)^c)$ and $\mathcal{L}(\Omega^3,(\Omega^1)^c)$ with the exteriors of the larger phases $\Omega^2$ and $\Omega^1$ ($j<i$).}\label{fig:interactions} 
\end{figure}
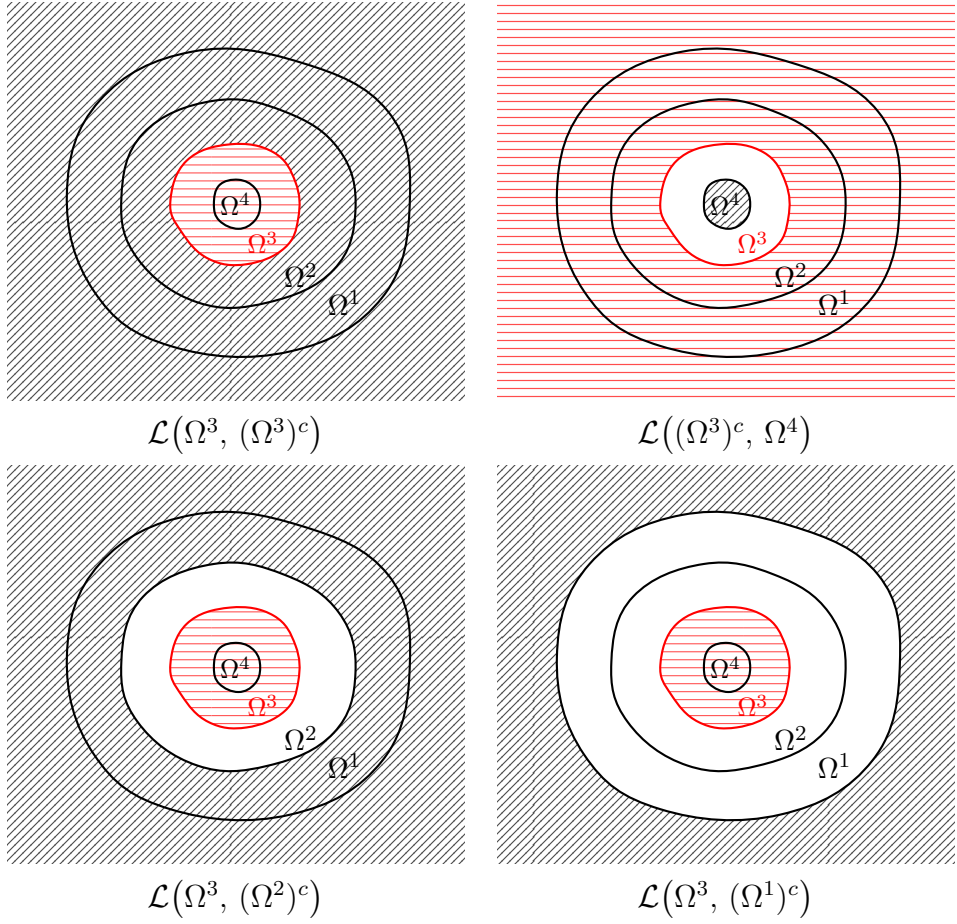

\subsection{Examples of the interacting flow}

Two simple symmetric configurations reduce system
\eqref{eq:velocity-intro} to finite-dimensional systems of ODEs:
nested parallel half-spaces and nested concentric balls.
For parallel half-spaces, the resulting system is the one governing
one-dimensional dislocation dynamics associated with the fractional
Allen--Cahn equation, previously studied in \cite{GonzalezMonneau, DipierroFigalliValdinoci,DipierroPalatucciValdinoci}. Global existence and uniqueness, together with
preservation of the ordering, were established in
\cite{GonzalezMonneau,VanMeursPeletierPozar}.
For concentric balls, the evolution reduces to a system of ODEs for
their radii. This system admits a unique maximal solution; the spheres
do not collide, and the solution exists until the innermost ball
shrinks to a point, which occurs in finite time. The corresponding
computations are provided in Appendix~\ref{sec:examples}.

\subsection{Strategies and prior work}\label{sec:significance}
The classical Allen--Cahn problem provides the local point of comparison. Its stationary limit was studied by Modica and Mortola \cite{MM}. Chen proved the appearance of motion by mean curvature for the evolutionary problem \cite{Chen}, and Evans, Soner and Souganidis established a level-set convergence framework that continues beyond smooth evolution \cite{EvansSonerSouganidis}. For fractional mean-curvature flow of a single set, Imbert \cite{Imbert09} developed a level-set formulation, providing a weak notion of evolution analogous to level-set approaches for classical front propagation \cite{Barles-Souganidis,EvansSonerSouganidis,EvansSpruck}. For fractional diffusion, Imbert and Souganidis developed a phase-field approach \cite{Imbert}. In the critical regime $s=1/2$, the multiple-loop analysis of Patrizi and Vaughan \cite{PatriziVaughan2} concerns a different scaling and limiting behavior. For $s<1/2$, the earlier paper \cite{HasaniPatrizi} proved convergence to fractional mean-curvature motion for a single front. The present paper addresses the additional question left by that result: what motion emerges when several well-separated transition layers evolve in the same strongly nonlocal equation?

Related one-dimensional results make the role of interaction particularly clear. When $n=1$, the analogue of \eqref{eq:velocity-intro} governs interacting interface points; parallel half-spaces in higher dimensions lead to an ODE of the same form, up to a dimensional coefficient (see Appendix~\ref{sec:examples}). Gonz\'{a}lez and Monneau studied particle dynamics at $s=1/2$ \cite{GonzalezMonneau}, while strongly nonlocal and less singular fractional regimes were investigated in \cite{DipierroFigalliValdinoci,DipierroPalatucciValdinoci}; related interacting-particle systems are considered in \cite{VanMeursPeletierPozar}. The behavior of differently oriented layers and their collisions is considered in \cite{MeursPatrizi,patval1,patval4}; long-time dynamics is studied in \cite{CozziDavilaDelPino,patval3}, and continuum limits involving many layers in \cite{patsan,patsan2}. Unlike one-dimensional transition points, the fronts in this paper are hypersurfaces: each has its own fractional curvature, while all the others contribute through nonlocal integrals. Work on fractional curvature flows and related geometric evolutions includes \cite{CesaroniDeLucaNovagaPonsiglione,CesaroniDipierroNovagaVal,CesaroniNovaga,ChambolleMoriniNovagaPonsiglione,ChambolleMoriniPonsiglione,CintiSinestrariValdinoci2,CintiSinestrariValdinoci,SaezVal}.  

The proof constructs global subsolutions and supersolutions to \eqref{eq:pde} around a prescribed smooth solution of \eqref{eq:velocity-intro}, and applies the comparison principle. This follows the barrier strategy developed for front propagation in \cite{Barles-DaLio,Barles-Souganidis} and its nonlocal extensions \cite{Imbert09,Imbert}. The single-front analysis \cite{HasaniPatrizi} supplies important ingredients, including the strongly nonlocal curvature approximation and estimates for the phase transition. Several new terms appear when fronts are superposed. Near a given boundary, the other layers generate an order-one nonlocal interaction; this is precisely the interaction term in \eqref{eq:velocity-intro}. The periodicity of $W$ permits an expansion around the layer at that boundary, but the resulting error must be controlled uniformly in the presence of all the other layers.

To achieve this control, we add a corrector near each front, governed by the operator obtained by linearizing the one-dimensional layer equation. The forcing of that corrector contains both the curvature contribution and the interactions with the other fronts. We also use a smooth extension of signed distance and auxiliary cutoffs to control the error away from every front, where an unmodified distance function need not be smooth. These features distinguish the construction from the critical-regime barriers in \cite{PatriziVaughan2} and extend the single-front strong-nonlocal construction in \cite{HasaniPatrizi}. We use the known properties of the one-dimensional profile and corrector \cite{DipierroFigalliValdinoci,MonneauPatrizi2}, together with estimates for fractional operators and comparison principles \cite{UsersGuide,JakobsenKarlsen,PatriziVaughan}. Section 3 explains the heuristic derivation  of the coupled velocity and corrector equation; Sections 5--7 turn that calculation into the convergence proof.

\subsection{Organization of the paper}

Section 2 recalls fractional mean curvature, expresses the coupled flow through signed distance functions, and constructs their smooth extensions. Section 3 derives the geometric law and the corrector equation formally. Section 4 collects facts about the fractional operator and the Allen--Cahn evolution. Section 5 introduces the phase transition, auxiliary functions and correctors, and states the estimates needed in the proof. Section 6 constructs the barriers, and Section 7 proves Theorem~\ref{thm:main_result} by comparison. Sections 8 and 9 prove the interaction and derivative estimates used in Section 5. Appendix A analyzes the parallel half-space and concentric-ball examples.

\subsection{Notations}
Throughout the paper, we denote by $C>0$ any  constant independent of $\ep$ and the parameters $\delta$, $\sigma$, and $R$,  which will be introduced later. 

We write $B(x_0,r)$ and $\overline{B}(x_0,r)$ for the open and closed balls of radius $r>0$ centered at $x_0\in\R^n$, respectively, and $S^{n-1}$ for the unit sphere in $\R^{n}$. 

For $\beta \in (0,1]$,  $k \in \N \cup \{0\}$ and $m\in\N$, we denote by $C^{k,\beta}(\R^m)$ the usual class of functions with bounded $C^{k,\beta}$ norm over $\R^m$. 
For $\beta=0$ we simply write $C^{k}(\R^m)$.
For multi-variable functions $v(\xi;t,x)$, we write $v \in C_\xi^{k,\beta}(\R)$ if $v(\cdot;t,x) \in C^{k,\beta}(\R)$ for all $t,x$ in the domain of $v$. 
Moreover, we use the dot notation for derivatives with respect to the variable $\xi$, namely $\dot{v}(\xi;t,x) = v_\xi(\xi;t,x)$.

Given a function $\eta = \eta(t,x)$, defined on a set $A$,  we write $\eta = O(\ep)$ if there is $C>0$ such that
$|\eta(t,x)| \leq C \ep$ for all $(t,x)\in A$, and we write $\eta = o_\ep(1)$ if  $\lim_{\ep \to 0} \eta(t,x) = 0$, uniformly in $(t,x)\in A$. 


For a set $A$, we denote by $\one_A$ the characteristic function of the set $A$. 

\section{Motion by fractional mean curvature}\label{sec:FMC}

In this section, we recall properties of fractional mean curvature, express the coupled front evolution in terms of signed distance functions, and describe smooth extensions of these functions away from the fronts.

\subsection{The fractional mean curvature} 

Let $\Omega$ be a smooth bounded subset of $\R^n$. For a point $x\in \partial \Omega$, the fractional mean curvature of order $2s$ of $ \Omega$ at $x$ is defined by
$$H_{2s}(\Omega)(x) =P.V.\int_{\R^n}\frac{\one_\Omega(z)-\one_{\Omega^c}(z)}{|z-x|^{n+2s}}\,dz,$$
where P.V.  denotes the Cauchy principal value. This quantity can also be expressed in terms of the signed  distance function $d$ to $\Omega$.
Indeed, since $$\Omega=\{z\,:\, d(z)>0\}$$ and using that 
$$P.V. \int_{\R^n}\frac{\one_{\{\nabla d(x)\cdot z>0\}}-\one_{\{\nabla d(x)\cdot z<0\}}}{|z|^{n+2s}}\,dz=0,$$ we can write
\begin{align*}
H_{2s}(\Omega)(x) &=P.V.\int_{\R^n}\frac{\one_\Omega(x+z)-\one_{\Omega^c}(x+z)}{|z|^{n+2s}}\,dz\\&
=P.V. \int_{\R^n}\frac{\one_{\{d(x+z)>0\}}-\one_{\{d(x+z)<0\}}+\one_{\{\nabla d(x)\cdot z<0\}}-\one_{\{\nabla d(x)\cdot z>0\}}}{|z|^{n+2s}}\,dz\\& 
=2 \int_{\{d(x+z)>0,\, \nabla d(x)\cdot z<0\}}\frac{dz}{|z|^{n+2s}}-2 \int_{\{d(x+z)<0,\, \nabla d(x)\cdot z>0\}}\frac{dz}{|z|^{n+2s}},
\end{align*}
where the last two integrals converge in the standard sense, as stated in Proposition \ref{prop:fmc_finite_result} below.
Assume $d$ is smooth in $Q_{2\rho}:=\{z\,:\,|d(z)|<2\rho\}$ for some $\rho>0$, then for $x\in Q_\rho$, define
\begin{equation}\label{kappa+de}\kappa^+[x,d]:=\int_{\{d(x+z)>d(x),\, \nabla d(x)\cdot z<0\}}\frac{dz}{|z|^{n+2s}},\quad 
\kappa^-[x,d] :=\int_{\{d(x+z)<d(x),\, \nabla d(x)\cdot z>0\}}\frac{dz}{|z|^{n+2s}},\end{equation}
and 
\begin{equation}\label{kappade}\kappa[x,d]:=\kappa^+[x,d]-\kappa^-[x,d].\end{equation}
From the discussion above, we obtain the identity
$$\kappa[x,d]=\frac12 H_{2s}(\{d>d(x)\})(x).$$ 
Roughly speaking, $\kappa[\cdot,d]$ plays the role of  $\Delta d$ in the local setting. 

A proof of the following result can be found, for instance, in \cite[Lemma 7.3]{PatriziVaughan}. 
 \begin{prop}\label{prop:fmc_finite_result}
    Assume $d$ of class $C^2(Q_{2\rho})$, then for all $x\in Q_\rho$,  the  quantities  $ \kappa^+[x,d] $ and $\kappa^-[x,d]$ are finite. 
\end{prop}
The fractional mean curvature of balls can be explicitly computed, see \cite[Lemma 2]{SaezVal} for a proof.
\begin{prop}\label{ballFMC}For $r>0$, let $d(x)=r-|x|$. Then, for $x\neq 0$, 
$$\kappa[x,d]=-\frac{\omega}{|x|^{2s}},$$
for some $\omega>0$.
\end{prop}
The following neighborhood regularity is a consequence of
\cite[Lemma 2.1]{CiraoloFigalliMaggiNovaga}, applied uniformly to the
parallel hypersurfaces $\{d=r\}$, and of the first-variation
formula for the fractional mean curvature in
\cite[Proposition B.2]{DavilaDelPinoWei}.
\begin{prop}\label{lem:kappa-neighborhood}
Let $\Omega\subset\mathbb R^n$ be a bounded open set with 
$C^{2,\alpha}$ boundary, for some $\alpha>2s$, and let $d$ be the signed distance to $\Omega$. Then there
exists $\rho>0$ such that the function
$$\kappa[x,d]=\frac12 H_{2s}(\{d>d(x)\})(x)$$ 
belongs to $C^1(Q_\rho)$.
\end{prop}
\subsection{Evolution of the  signed distance functions}
Let  $(\Omega_t^i)_{i=1}^N$, $t \in [0, T]$, be a family of sets with smooth boundaries 
$\Gamma^i_t=\partial\Omega_t^i $. For each $i=1,\ldots,N$, let 
$d_i(t,x)$ denote the signed distance function to $\Omega_t^i$, defined by 
\begin{align*}
d_i(t, x) =
\begin{cases}
d(x, \Gamma^i_t) & \text{for } x\in \Omega_t^i \\
-d(x, \Gamma^i_t) & \text{for } x\in (\Omega_t^i)^c.
\end{cases}
\end{align*}
It is well known that $d_i$ is smooth in a tubular neighborhood of $\Gamma_i^t$, for every $t\in[0,T]$. 

Suppose the sets $(\Omega_t^i)_{i=1}^N$   satisfying the nesting condition $\Omega_{t}^{i+1} \subset \subset  \Omega_t^{i}$, 
evolve with outward normal velocity $v_i$, given by
\begin{equation*}
v_i(t,x)=c_0\left(\frac{1}{2} H_{2s}(\Omega^i_t)(x)-\sum_{j=1}^{i-1} \int_{(\overline{\Omega_t^j})^c} \frac{dz}{|z-x|^{n+2s}}  +\sum_{j=i+1}^{N} \int_{\Omega^j_t} \frac{dz}{|z-x|^{n+2s}}+\sigma\right),
\end{equation*}
for $i=1,\ldots, N$, $t\in[0,T]$, $ x\in \Gamma_t^i$, and 
some $\sigma = \sigma(t,x)$.  Then, since 
$$\Omega_t^i=\{x\in\R^n\,:\,d_i(t,x)>0\},\qquad (\overline{\Omega_t^i})^c=\{x\in\R^n\,:\,d_i(t,x)<0\},$$
$$\partial_t d_i(t,x) = v_i(t,x),\qquad \kappa[x,d_i(t,\cdot)]=\frac12 H_{2s}(\{d_i>d_i(t,x)\})(x),\qquad x\in \Gamma_t^i,$$
and by the nesting assumption,
\[
d_j(t,x)>0 \ \text{for } j<i,
\qquad
d_j(t,x)<0 \ \text{for } j>i,
\qquad x\in\Gamma_t^i,
\]
it follows that the signed distance functions $d_i$ satisfy, for $i=1,\ldots, N$,   $t\in[0,T]$, and  $ x\in \Gamma_t^i$,
\begin{equation}\label{eq:fmc_interaction}
\partial_t d_i = c_0\Bigg(\kappa\!\big[x, d_i(t,\cdot)\big] -\sum_{j \neq i}\operatorname{sgn}(d_j(t,x)) \int_{\{d_j(t,z) \operatorname{sgn}(d_j(t,x)) < 0\}} \frac{dz}{|z-x|^{n+2s}} +\sigma\Bigg).
\end{equation}

\subsection{Smooth extension of the signed distance function}

Recall that the signed distance function $d = d(t,x)$ associated with a front $\Gamma_t$  is smooth in some neighborhood of the front, provided $\Gamma_t$ is smooth. 
However, in general, $d$ is not smooth away from the front. 
Throughout the paper, we shall  use the following smooth extension of the distance function away from $\Gamma_t$.

\begin{defn}[Extension of the signed distance function]\label{defn:extension}
For $t\in [0,T]$, let $\tilde{d}$ be the signed distance function from a bounded domain  $\Omega_t$ with boundary  $\Gamma_t$  and let   $\rho>0$ be such that $\tilde{d}(t,x)$  is smooth in 
 \[
Q_{3\rho} := \{(t,x)\in [0,T] \times \R^n: |\tilde{d}(t,x)| < 3\rho\}.
\]
 Since the distance function $\tilde d$ is globally Lipschitz in $(t,x)$, we may choose a smooth approximation $\tilde d_\tau$ of $\tilde d$,  depending on a parameter $\tau>0$, such that 
\begin{equation*}
\tilde d_\tau=\tilde d+o_\tau(1),\qquad |\partial_t \tilde d_\tau|,\, |\nabla \tilde d_\tau|\le C, \qquad |D^2 \tilde d_\tau|=O(\tau^{-1}). 
\end{equation*} 
Let $\eta(t, x)$ be a smooth, bounded function such that 
\[
\eta = 1~\hbox{in}~\{|\tilde{d}|\leq 2\rho\}, \qquad \eta = 0~\hbox{in}~\{|\tilde{d}|\geq 3\rho\},\qquad 0 \leq \eta \leq 1.
\]
We extend $\tilde{d}(t,x)$ on the set $\{(t,x)\in [0,T] \times \R^n: |\tilde{d}(t,x)|\geq 2\rho\}$ to  the smooth  function $d(t,x)$ given by
\[
d(t,x) 
:= \begin{cases}
\tilde{d}(t,x) & \hbox{in}~Q_\rho=\{|\tilde{d}(t,x)|<2 \rho\}\\
\tilde{d}(t,x) \eta(t,x) + \tilde d_\tau(t,x)(1-\eta(t,x)) & \hbox{in}~\{2\rho \leq  |\tilde{d}(t,x)| \leq 3 \rho\}\\
\tilde d_\tau(t,x) &\hbox{in}~\{|\tilde{d}(t,x)|> 3\rho\}.
\end{cases}
\]
Notice that, 
\begin{equation}\label{approxdtau}
d=\tilde d+o_\tau(1),\qquad |\partial_t d|,\, |\nabla d|\le C, \qquad |D^2 d|=O(\tau^{-1}). 
\end{equation} 
\end{defn} 
\begin{rem}\label{Kdextensionrem}
By the definition of $ d$, for $\tau$ sufficiently small  and  $(t,x)\in Q_\rho$, we have
 $$\partial_t d (t,x)=\partial_t\tilde d (t,x),\quad\nabla d (t,x)=\nabla \tilde d(t,x),$$
 $$\{z\,:\, d(t, x+z)> d(t, x)\}=\{z\,:\,\tilde d(t, x+z)>\tilde d(t, x)\},$$
 $$\{z\,:\, d(t, x+z)< d(t, x)\}=\{z\,:\,\tilde d(t, x+z)<\tilde d(t, x)\}.$$
Consequently,  for $(t,x)\in Q_\rho$,
$$\kappa[x, d(t,\cdot)]=\kappa[x, \tilde  d(t,\cdot)].$$
\end{rem} 

\section{Heuristics}\label{sec:Heuristics}

Here, we give two formal computations relating to Theorem \ref{thm:main_result} and its proof.
We use the notation $\simeq$ to denote equality up to adding terms that vanish as $\ep \to 0$.

\subsection{Derivation of the flow \eqref{eq:fmc_interaction}}\label{sec:heuristicspart1}
For simplicity, we restrict our discussion to the case $N=2$. 
For the following formal computations, assume that the signed distance functions $d_1(t,x)$ and $d_2(t,x)$ associated to $\Omega_t^1$ and $\Omega_t^2$ are smooth and satisfy $\abs{\nabla d_1} = \abs{\nabla d_2} = 1$.
Moreover, we assume that there is a positive, uniform distance $2\ell_0$ between $\Gamma_t^1$ and $\Gamma_t^2$.

Consider the following ansatz for the solution of \eqref{eq:pde}-\eqref{initial_data}
\begin{equation} \label{eq:ansatz1}
u^{\ep}(t,x) \simeq \phi\( \frac{d_1(t,x)}{\ep}\) + \phi\( \frac{d_2(t,x)}{\ep}\),
\end{equation}
with $\phi$ the solution of \eqref{eq:standing wave}.
Plugging the ansatz into \eqref{eq:pde}, the left-hand side becomes
\begin{equation}\label{eq:antatz1-left}
\ep\partial_t u^{\ep}
 \simeq \dot{\phi}\(\frac{d_1}{\ep}\) \partial_t d_1 + \dot{\phi}\(\frac{d_2}{\ep}\) \partial_t d_2.
\end{equation}
On the other hand, we use the equation for $\phi$ in \eqref{eq:standing wave} to compute the fractional Laplacian of the ansatz:
\begin{equation}\label{eq:antatz1-right}
\begin{aligned}
 \mathcal{I}^s _n [u^{\ep}]
 &\simeq \mathcal{I}^s _n\bigg[\phi\(\frac{d_1}{\ep}\)\bigg] + \mathcal{I}^s _n\bigg[\phi\(\frac{d_2}{\ep}\)\bigg] \\
 &= \( \mathcal{I}^s_n\bigg[\phi\(\frac{d_1}{\ep}\)\bigg]  - \frac{C_{n,s}}{\ep^{2s}}\mathcal{I}^s_1[\phi]\(\frac{d_1}{\ep}\)  \)  +\frac{C_{n,s}}{\ep^{2s}}\mathcal{I}^{s}_1[\phi]\(\frac{d_1}{\ep}\)  + \mathcal{I}^s _n\bigg[\phi\(\frac{d_2}{\ep}\)\bigg]\\
 & =\bar a_{\ep}[d_1] +\frac{1}{\ep^{2s}}W'\(\phi\( \frac{d_1}{\ep}\)\) + P_\ep[d_2],
\end{aligned}
\end{equation}
where we denote 
\begin{equation}\label{oldaep}
\bar a_{\ep}[d](t,x):=\mathcal{I}^s_n\bigg[\phi\(\frac{d(t,\cdot)}{\ep}\)\bigg] (x) - \frac{C_{n,s}}{\ep^{2s}}\mathcal{I}^s_1[\phi]\(\frac{d(t,x)}{\ep}\),
\end{equation}
with $C_{n,s}>0$ given in \eqref{eq:Cns}, and 
\begin{equation}\label{eq:P_ep_heuristic}
P_\ep[d](t,x):= \mathcal{I}_n^s\left[\phi\left(\frac{d(t,\cdot)}{\ep}\right)\right](x).
\end{equation}
By Lemma \ref{lem:1 to n facrional Laplacian}, applied to $v=\phi$ with $e=\nabla d_1(t,x)$, we have
\begin{equation*}\frac{C_{n,s}}{\ep^{2s}}\mathcal{I}^s_1[\phi]\(\frac{d_1(t,x)}{\ep}\)=\frac{1}{\ep^{2s}}\int_{\R^n}\(\phi\(\frac{d_1(t,x)}{\ep}+\nabla d_1(t,x)\cdot z\)-\phi\(\frac{d_1(t,x)}{\ep}\)\)\frac{dz}{|z|^{n+2s}}.\end{equation*}
Hence, since
\begin{equation*}\mathcal{I}^s_n\bigg[\phi\(\frac{d_1(t,\cdot)}{\ep}\)\bigg] (x)=\int_{\R^n} \(\phi\(\frac{d_1(t,x+ z)}{\ep}\)-\phi\(\frac{d_1(t,x)}{\ep}\)\)\frac{dz}{|z|^{n+2s}},\end{equation*}
after a change of variables, we can write $\bar a_\ep[d_1]$ as follows
\begin{align*}\bar a_\ep[d_1](t,x)=\int_{\R^n} \(\phi\(\frac{d_1(t,x+ z)}{\ep}\)-\phi\(\frac{ d_1(t,x) + \nabla d_1(t,x)\cdot z}{\ep}\)\)\frac{dz}{|z|^{n+2s}}.
\end{align*}
Now freeze a point $(t,x)$ such that $x$ is near the  first front $\Gamma_t^1$, and let $\xi := d_1(t,x)/\ep$. Since $d_1$ grows linearly away from $\Gamma_t^1$, we can assume, at least formally, separation of scales. That is, assume that $\xi \in \R$ and $(t,x)$ are independent variables. Furthermore, since the two interfaces are separated by a distance of at least $2\ell_0$, we have
 $|d_2(t,x)| \geq \ell_0$.
By \cite[Lemma 5.2]{HasaniPatrizi} (see Lemma \ref{thm: b_ep fmc} for a precise statement), it follows that
\begin{equation}\label{aepconvergenceoldthm}\bar a_\ep[d_1](t,x)\simeq \kappa[x,d_1(t,\cdot)].
\end{equation}
Moreover, by Lemma \ref{lem:interaction asymptotics}, 
\begin{equation}\label{PepsimP_heur}
P_\ep[d_2](t,x) \simeq P[d_2](t,x), 
\end{equation}
where
\begin{equation}\label{def:P}
P[d](t,x) := -\operatorname{sgn}(d(t,x)) \int_{\{d(t,z) \operatorname{sgn}(d(t,x)) < 0\}} \frac{dz}{|z-x|^{n+2s}}.
\end{equation}
Using these results and assuming separation of scales, we can now complete our formal derivation. Since the ansatz $u^{\ep}$ approximates a solution of \eqref{eq:pde}, we can multiply the equation by $\dot{\phi}(\xi)$ and integrate over $\xi \in \R$, obtaining
\begin{equation}\label{eq:freeze}
\int_{\R} \ep \partial_t u^{\ep} \dot{\phi}(\xi) \, d \xi \simeq \int_{\R} \( \mathcal{I}^s_n[ u^{\ep}]  - \frac{1}{\ep^{2s} } W'(u^\ep) \) \dot{\phi}(\xi) \, d \xi.
\end{equation}
For convenience, we will consider the left and right-hand sides separately again.

First, the left-hand side of \eqref{eq:freeze} with \eqref{eq:antatz1-left} gives
\begin{align*}
\int_{\R} \ep \partial_t u^{\ep}(t,x)\, \dot{\phi}(\xi) \, d \xi
 &\simeq \partial_t d_1(t,x) \int_{\R} [\dot{\phi}(\xi)]^2  \, d \xi + \partial_t d_2(t,x) \int_{\R} \dot{\phi}\(\frac{d_2(t,x)}{\ep}\) \dot{\phi}(\xi) \, d \xi\\
 &\simeq c_0^{-1}\partial_t d_1(t,x)+ O\(\frac{\ep^{2s+1}}{\ell_0^{2s+1}}\)\int_{\R}  \dot{\phi}(\xi) \, d \xi\\
 &\simeq  c_0^{-1}\partial_t d_1(t,x),
\end{align*}
where we used the asymptotics on $\dot{\phi}$ in  \eqref{eq:asymptotics for phi dot} and that $|d_2(t,x) | \geq \ell_0$,  and where
\begin{equation}\label{def:c0}
c_0^{-1} := \int_\R [\dot{\phi}(\xi)]^2 \, d \xi.
\end{equation}
Next,  we look at the right-hand side of \eqref{eq:freeze}. 
Denote by $\tilde\phi:=\phi-H$, where $H$ is the Heaviside function. 
Using the periodicity of $W$ and performing a Taylor expansion around $\phi\(\frac{d_1}{\ep}\)$, we get
\begin{equation*}\label{eq:nonlinearity_heuristic}
\begin{split}
\frac{1}{\ep^{2s}}W'(u^\ep)&= \frac{1}{\ep^{2s}}W'\(\phi\(\frac{d_1}{\ep}\)+\tilde\phi\(\frac{d_2}{\ep}\)\)\\&
\simeq \frac{1}{\ep^{2s}}W'\(\phi\(\frac{d_1}{\ep}\)\) +\frac{1}{\ep^{2s}}W''\(\phi\(\frac{d_1}{\ep}\)\)\tilde\phi\(\frac{d_2}{\ep}\).
\end{split}
\end{equation*}
Substituting this expression into the righ-hand side of \eqref{eq:freeze} and using \eqref{eq:antatz1-right}, the term $\frac{1}{\ep^{2s}}W'\(\phi\(\frac{d_1}{\ep}\)\)$ cancels,  and we obtain 
\begin{align*}
\int_{\R}  \left[\mathcal{I}^s_n [u^{\ep}(t,\cdot)](x) - \frac{1}{\ep^{2s} }W'(u^{\ep}(t,x))\right] \dot{\phi}(\xi) \, d \xi &\simeq \int_{\R} \left( \bar a_{\ep}[d_1](t,x) + P_\ep[d_2](t,x) \right) \dot{\phi}(\xi)\, d \xi\\&
 +\frac{1}{\ep^{2s}}\int_{\R}\tilde\phi\(\frac{d_2(t,x) }{\ep}\)W''(\phi(\xi))\dot{\phi}(\xi) \, d \xi.
\end{align*}
Using the facts that $\phi( -\infty) = 0$,  $\phi(\infty)=1$, $W'(0)=W'(1)$,  $|d_2(t,x)|\geq\ell_0$, 
and the asymptotic behavior of $\phi$ given in \eqref{eq:asymptotics for phi}, we obtain
\begin{align*}
    \frac{1}{\ep^{2s}}  \int_{\R}\tilde\phi\(\frac{d_2(t,x) }{\ep}\)W''(\phi(\xi))\dot{\phi}(\xi) \, d \xi = \frac{1}{\ep^{2s}}O\left(\frac{\ep^{2s}}{\rho^{2s}}\right)  \int_{\R} \frac{d}{d\xi} \left[W'(\phi(\xi)) \right] \, d \xi = 0.
\end{align*}

\noindent
Combining the above estimates, \eqref{eq:freeze} for the ansatz gives
\begin{align*}
c_0^{-1}\partial_t d_1(t,x) \simeq  \bar{a}_\ep[d_1](t,x) + P_\ep[d_2](t,x).
\end{align*}
Using \eqref{aepconvergenceoldthm} and \eqref{PepsimP_heur}, we then obtain
\begin{align*}
\partial_t d_1(t,x) \simeq c_0 \left(\kappa[x,d_1] + P[d_2](t,x)\right) \quad  \text{near}~\Gamma_t^1.
\end{align*}
The computation for $x$ near $\Gamma_t^2$ is similar.
We therefore conclude that each interface evolves by fractional mean curvature together with an interaction term generated by the other interface:
\[
\begin{cases}
\displaystyle \partial_t d_1(t,x) \simeq c_0 \left(\kappa[x,d_1] + P[d_2](t,x)\right) & \hbox{near}~\Gamma_t^1\\[2ex]
\displaystyle \partial_t d_2(t,x) \simeq c_0 \left(\kappa[x,d_2] + P[d_1](t,x)\right) & \hbox{near}~\Gamma_t^2.
\end{cases}
\]

\subsection{Derivation of equation \eqref{eq:linearized wave}}\label{Ansatzsection}

It is actually necessary to add a lower-order correction to \eqref{eq:ansatz1} for the ansatz to solve the fractional Allen--Cahn equation \eqref{eq:pde}. This was already observed in the one-dimensional case in \cite[Section 3.1]{GonzalezMonneau}. The correction involves, for each front, a function $\psi_i = \psi_i(\xi;t,x)$ belonging to the space
$\{v\in H^s(\R)\,:\,\int_\R v(\xi)\dot\phi(\xi)\,d\xi=0\}$, which satisfies an equation in the variable $\xi$ involving the linearized operator $\mathcal{L}$ associated with \eqref{eq:standing wave} around $\phi$, defined by
\begin{equation}\label{eq:linearized operator heuristic}
\mathcal{L}[\psi] = -C_{n,s}\mathcal{I}^s_1[\psi] + {W}''(\phi) \psi.
\end{equation}
To derive the equation satisfied by $\psi_i$, we follow the approach of \cite[Section 3]{HasaniPatrizi}, where the computation is carried out in the case $N=1$. The presence of multiple interfaces, however, introduces an additional interaction term, and consequently the resulting equation differs slightly from the one obtained there. As before, for simplicity, we restrict our discussion to the case $N=2$. We also assume that the fronts $\Gamma_t^1$ and $\Gamma_t^2$ remain separated by a uniform positive distance $2\ell_0$.

In order to showcase the equation for the corrector, for $\sigma\in\R$, let $v^\ep$ be the solution to
\begin{equation}\label{eq:pde-sub}
\ep \partial_t v^{\ep} = \mathcal{I}^s_n v^{\ep}  -\frac{1}{\ep^{2s}} W'(v^\ep)- \sigma.
\end{equation}

Assume that the signed distance functions $d_1(t,x)$ and $d_2(t,x)$ associated to $\Omega_t^1$ and $\Omega_t^2$ are smooth, satisfy $|\nabla d_1| = | \nabla d_2|=1$ and 
\begin{equation}\label{eq:flow_heuristic2}
    \begin{cases}
        \partial_t d_1 = c_0 (\kappa[x,d_1] + P[d_2] -\sigma) & \hbox{near}~\Gamma_t^1\\[2ex] 
        \partial_t d_2 = c_0 (\kappa[x,d_2] + P[d_1] -\sigma) & \hbox{near}~\Gamma_t^2.
    \end{cases}
\end{equation}
Consider the new corrected ansatz
\begin{equation}\label{eq:ansatz_two_fronts}
\begin{split}
v^{\ep}(t,x) \simeq &\phi\( \frac{d_1(t,x)}{\ep}\) + \phi\( \frac{d_2(t,x)}{\ep}\) + \ep^{2s} \(\psi_1\( \frac{d_1(t,x)}{\ep};t,x\) + \psi_2\( \frac{d_2(t,x)}{\ep};t,x\) \right.\\&\left.+ w(t,x)\)
\end{split}
\end{equation}
where $\psi_i = \psi_i(\xi; t,x)$ and $w(t,x)$ are functions to be determined.

Assume that $x$ is near the front $\Gamma_t^1$. Then, as before, we have $|d_2(t,x)|\geq\ell_0$. 
Moreover, recalling the definitions of $\bar a_\ep[d]$ and $P_\ep[d]$ in \eqref{oldaep} and \eqref{eq:P_ep_heuristic}, respectively, the estimates \eqref{aepconvergenceoldthm} and \eqref{PepsimP_heur} remain valid.
 Plugging the ansatz into \eqref{eq:pde-sub}, the left-hand side gives
\begin{equation}\label{eq:ansats2-left_two2}
\ep \partial_t v^{\ep}
 \simeq  \dot{\phi}\(\frac{d_1}{\ep}\)\partial_t d_1,
\end{equation}
where,  as before, we used the asymptotic estimate \eqref{eq:asymptotics for phi dot} and the fact that $|d_2(t,x)|\geq\ell_0$ to get  $\dot\phi(d_2(t,x)/\ep)\simeq0$, and we assume for each $i=1,2$
\begin{equation}\label{timederivest_heursec}
\ep^{2s}   (\dot{\psi}_i \partial_t d_i+\ep \partial_t \psi_i+\ep \partial_t w)\simeq 0.
\end{equation}
For the fractional Laplacian term, we compute
\begin{equation}\label{eq:In_expansion}
\begin{aligned}
\mathcal{I}^s_n[v^\ep]
 &\simeq \mathcal{I}^s_n\left[ \phi\( \frac{d_1}{\ep}\) \right] + \mathcal{I}^s_n\left[ \phi\( \frac{d_2}{\ep}\) \right]
  + \ep^{2s} \mathcal{I}^s_n[\psi_1] + \ep^{2s} \mathcal{I}^s_n[\psi_2] + \ep^{2s} \mathcal{I}_n^sw.
\end{aligned}
\end{equation}
Adding and subtracting $\frac{C_{n,s}}{\ep^{2s}}\mathcal{I}^s_1[\phi]\(\frac{d_1}{\ep}\)$ and  $C_{n,s} \mathcal{I}^s_1[\psi_i]\( \frac{d_i}{\ep}\)$, using  \eqref{eq:standing wave},  and assuming
\begin{equation}
\label{psilemmaheuristics}
\ep^{2s} \mathcal{I}^s_n\left[ \psi_i\( \frac{d_i(t,\cdot)}{\ep};t,\cdot\) \right]-C_{n,s} \mathcal{I}^s_1[\psi_i]\( \frac{d_i}{\ep}\) \simeq 0, \quad \ep^{2s} \mathcal{I}_n^sw \simeq 0
\end{equation}
we find (recall \ref{eq:P_ep_heuristic})
\begin{equation}\label{eq:ansats2-right1_two2}
\begin{aligned}
\mathcal{I}^s_n[v^\ep]&\simeq \bar a_{\ep}[d_1] + P_{\ep}[d_2]
  +  \frac{1}{\ep^{2s}}W'\(\phi\(\frac{d_1}{\ep}\)\)+C_{n,s} \mathcal{I}^s_1[\psi_1]\( \frac{d_1}{\ep}\)+C_{n,s} \mathcal{I}^s_1[\psi_2]\( \frac{d_2}{\ep}\).
\end{aligned}
\end{equation}
Performing   a Taylor expansion of $W'$ around $\phi(d_1/\ep)$, we  obtain
\begin{equation}\label{eq:ansats2-right2_two2}
\begin{aligned}
\frac{1}{\ep^{2s}}W'(v^{\ep})
 &\simeq \frac{1}{\ep^{2s}} \left[W'\(\phi\(\frac{d_1}{\ep}\)\)+ W''\(\phi\(\frac{d_1}{\ep}\)\) \(\phi\(\frac{d_2}{\ep}\) + \ep^{2s} \psi_1 + 
 \ep^{2s} \psi_2+\ep^{2s}  w\)\right].
\end{aligned}
\end{equation}
Assuming that
\begin{equation}\label{psiinfheur}
\psi_i(\pm\infty;t,x)=0,\qquad \mathcal{I}^s_1[\psi_i(\cdot;,t,x)](\pm\infty)=0,
\end{equation}
 since  $|d_2(t,x)|\geq\ell_0$, we may neglect  the terms involving  $\psi_2$ in \eqref{eq:ansats2-right1_two2} and \eqref{eq:ansats2-right2_two2}. Substituting \eqref{eq:ansats2-left_two2},  \eqref{eq:ansats2-right1_two2}, and \eqref{eq:ansats2-right2_two2} into \eqref{eq:pde-sub} we therefore obtain
\begin{equation}\label{eq:first corrector eqn_two}
\begin{aligned}
 \dot{\phi}\( \frac{d_1}{\ep}\) \partial_td_1
  &\simeq \bar a_{\ep}[d_1] + P_\ep[d_2] +C_{n,s} \mathcal{I}^s_1[\psi_1]\( \frac{d_1}{\ep}\)\\
 &\quad- \frac{1}{\ep^{2s}}W''\(\phi\( \frac{d_1}{\ep}\)\)\phi\(\frac{d_2}{\ep}\)- W''\(\phi\( \frac{d_1}{\ep}\)\)   \psi_1\\
 &\quad-W''\( \phi \(\frac{d_1}{\ep}\)\)w(t,x)-\sigma.
\end{aligned}
\end{equation}
By the evolution equation for $d_1$ in \eqref{eq:flow_heuristic2}, together with \eqref{aepconvergenceoldthm} and \eqref{PepsimP_heur}, we have 
$$\partial_td_1\simeq c_0(\bar a_{\ep}[d_1]+P_\ep[d_2]-\sigma).$$
Inserting this expression into \eqref{eq:first corrector eqn_two}, and rearranging terms (recall the definition of $\mathcal{L}$ in \eqref{eq:linearized operator heuristic}),  we find
\begin{align*}
\mathcal{L}[\psi_1]\( \frac{d_1}{\ep}\)
 &\simeq (\bar a_{\ep}[d_1]+P_\ep[d_2]-\sigma)\(1- c_0\dot{\phi}\( \frac{d_1}{\ep}\)\)\\
 &\quad - \frac{1}{\ep^{2s}}W''\(\phi\( \frac{d_1}{\ep}\)\)\phi\(\frac{d_2}{\ep}\)  -W''\(\phi\( \frac{d_1}{\ep}\)\)w(t,x).
\end{align*}
Evaluating the equation when $|d_1(t,x)|\gg \ep$, and using 
 the asymptotic behavior of  $\phi$ and $\dot\phi$ in  \eqref{eq:asymptotics for phi} and  \eqref{eq:asymptotics for phi dot}, the fact  that $W''(0)=W''(1)$, and \eqref{psiinfheur}, we obtain
 $$ \mathcal{L}[\psi_1]\( \frac{d_1}{\ep}\)\simeq 0,\quad \dot{\phi}\( \frac{d_1}{\ep}\)\simeq0,\quad W''\(\phi\( \frac{d_1}{\ep}\)\)\simeq W''(0).$$
Consequently, $w$ must satisfy
$$W''(0)w(t,x) = \bar a_{\ep}[d_1](t,x) + P_\ep[d_2](t,x) -\sigma - \frac{W''(0)}{\ep^{2s}}\phi\(\frac{d_2}{\ep}\).$$
Substituting this into the earlier expression, the terms involving $\phi(d_2/\ep)$ cancel, yielding the following equation  for $\psi_1=\psi_1(\xi;t,x)$ in the variable $\xi$,
\begin{equation}\label{psiequationheuristic}\mathcal{L}[\psi_1](\xi) =\left( c_0\dot{\phi}(\xi)
  + \frac{W''\left( \phi (\xi)\right) - W''(0)}{ W''(0)} \right)\left(\sigma -\bar{a}_\ep[d_1](t,x) - P_\ep[d_2](t,x)\right).\end{equation}
This is the corrector equation associated with the first front. An analogous computation near the front $\Gamma_t^2$ yields the same equation for $\psi_2$, with the roles of $d_1$ and $d_2$ interchanged.
Compared with the single-front case derived in \cite[Section 3]{HasaniPatrizi}, the only difference is the presence of the interaction term $P_\ep[d_2]$ on the right-hand side.

The existence of correctors satisfying \eqref{timederivest_heursec}, \eqref{psilemmaheuristics}, \eqref{psiinfheur}, and \eqref{psiequationheuristic} is established in Lemmas \ref{lem:psi-reg} and \ref{lem:ae psi estimate}.

As for the ansatz for $v^\ep$, using that $W'(0)=0,$ the definitions of $\bar a_\ep[d]$ and $P_\ep[d]$ in \eqref{oldaep} and \eqref{eq:P_ep_heuristic}, and the equation for $\phi$ in \eqref{eq:standing wave}, we have
\begin{align*}
P_\ep[d_2]  - \frac{W''(0)}{\ep^{2s}}\phi\(\frac{d_2}{\ep}\)
&\simeq P_\ep[d_2]-\frac{1}{\ep^{2s}}W'\(\phi\(\frac{d_2}{\ep}\)\)= P_\ep[d_2]- \frac{C_{n,s}}{\ep^{2s}}\mathcal{I}^s_1[\phi]\(\frac{d_2}{\ep}\)=\bar{a}_\ep[d_2].
\end{align*}
Substituting this into the earlier expression for  $w$, from \eqref{eq:ansatz_two_fronts}  we obtain  
\begin{equation*}
\begin{split}
v^{\ep}(t,x) \simeq& \phi\( \frac{d_1(t,x)}{\ep}\) + \phi\( \frac{d_2(t,x)}{\ep}\) + \ep^{2s} \(\psi_1\( \frac{d_1(t,x)}{\ep};t,x\) + \psi_2\( \frac{d_2(t,x)}{\ep};t,x\)\right.\\&\left. +\frac{1}{W''(0)}(\bar a_\ep[d_1](t,x)+\bar a_\ep[d_2](t,x)-\sigma) \).
\end{split}
\end{equation*}
This final corrected ansatz will be used to construct subsolutions and supersolutions to  \eqref{eq:pde}, depending on the sign of $\sigma$ (see  Section~\ref{barriersection}).
In order to construct global in space subsolutions and supersolutions, it will be necessary to modify the definition of $\psi_i$ and $\bar a_\ep $ to control additional error terms that arise far from the interfaces $\Gamma^i_t$. As in \cite{HasaniPatrizi}, this will involve the auxiliary function $\mu$
 defined in \eqref{def:mu_definition} and the approximation of 
 $\bar{a}_\ep$ given in \eqref{aepsilondef}.  

\section{Preliminary results on the fractional Laplacian}\label{sec:ae-frac}
In this section, we recall a few basic properties of the operator $\mathcal{I}^s_n $, which will be used later in the paper.
Let $u\in C^{0,1}(\R^n)$.  Then, 
for any $R>0$, we can write
\begin{align*}
\mathcal{I}^s_n u(x) 
 &= \int_{\{\abs{z}<R\}} \left( u(x+z) - u(x)\right) \,\frac{dz}{\abs{z}^{n+2s}}
  + \int_{\{\abs{z}>R\}} \left( u(x+z) - u(x)\right) \,\frac{dz}{\abs{z}^{n+2s}}.
\end{align*}
Both integrals above are finite, and we can bound $\mathcal{I}^s_n u$ as follows,

\begin{equation}\label{Iubound-C11_bis}
|\mathcal{I}^s_n u(x)|\leq C\left(\|Du\|_\infty R^{1-2s}+\frac{\|u\|_\infty}{R^{2s}}\right), 
\end{equation}
where $C>0$ is a constant depending only on $n$ and $s$. 
The estimate \eqref{Iubound-C11_bis} follows from the following lemma, whose proof is a straightforward computation in polar coordinates.

\begin{lem}\label{kernellemma3}
There exist $C_1,\, C_2>0$ such that for any $R>0$,
$$\int_{\{|z|<R\}}\frac{dz}{|z|^{n+2s-1}}= C_1 R^{1-2s}
\quad \hbox{and} \quad \int_{\{|z|>R\}}\frac{dz}{|z|^{n+2s}}= \frac{C_2}{R^{2s}}.$$
\end{lem}
We will frequently use Lemma \ref{kernellemma3}  throughout the paper without further reference.

We will also need   the following result, which provides a representation of  the one-dimensional fractional Laplacian of a function defined on $\mathbb{R}$ as an $n$-dimensional fractional Laplacian. 

\begin{lem} \cite[Lemma 3.2]{PatriziVaughan}\label{lem:1 to n facrional Laplacian}
For a vector $e \in \R^n$ and a function $v\in C^{1,1}(\R)$, let $v_e(x) = v(e\cdot x): \R^n \to \R$. Then,
\[
\mathcal{I}_n^s[v_e](x) =|e|^{2s} C_{n,s} \mathcal{I}_1^s[v](e\cdot x)
\]
where 
\begin{equation}\label{eq:Cns}
C_{n,s} =  \int_{\R^{n-1}} \frac{1}{(\abs{z}^2 + 1)^{\frac{n+2s}{2}} } \, dz.
\end{equation}
Consequently,
\begin{equation*}\label{eq:one to n}
|e|^{2s} C_{n,s} \mathcal{I}_1^s[v](\xi) = \int_{\R^n} \left( v(\xi + e \cdot z) - v(\xi)\right) \frac{dz}{\abs{z}^{n+2s}}, \quad \xi \in \R.
\end{equation*}
\end{lem}

\subsection{Properties of solutions to \eqref{eq:pde}}

Here, we state existence, uniqueness, and comparison principles for viscosity solutions to \eqref{eq:pde} for a fixed $\ep>0$. 

First, the following comparison principle can be found in \cite{JakobsenKarlsen} and will be used throughout the paper without reference. 
For the definition of viscosity subsolutions, supersolutions, and solutions, see also \cite{UsersGuide}.
For ease, we denote by $USC_b([0,T] \times \R^n)$ (resp.~$LSC_b([0,T] \times \R^n)$) the set of upper (resp.~lower) semicontinuous functions on $[0,T] \times \R^n$ which are bounded on $[0,T] \times \R^n$. 

\begin{prop}[Comparison principle in $\R^n$]
Fix $\ep>0$. 
If $u \in USC_b([0,T] \times \R^n)$ is a viscosity subsolution and $v \in LSC_b([0,T] \times \R^n)$ is a viscosity supersolution of \eqref{eq:pde} such that $u(0,\cdot) \leq v(0,\cdot)$
on $\R^n$, then $u \leq v$ on $[0,T] \times \R^n$. 
\end{prop}

We next state the following existence and uniqueness result, whose proof can be found in \cite{HasaniPatrizi}.

\begin{prop}\cite[Proposition 4.5]{HasaniPatrizi}
Fix $\ep>0$ and let $u_0 \in C^{0,1}(\R^n)$. 
There exists a unique viscosity solution $u^\ep \in C([0,\infty) \times \R^n) \cap L^{\infty}([0,\infty)\times \R^n)$  to \eqref{eq:pde} with initial datum $u^\ep(0,x) = u_0(x)$. 
\end{prop}

\section{The phase transition, the corrector, and the auxiliary function}\label{sec:corrector-section}
In this section, we introduce the phase transition $\phi$ and the correctors $\psi_i$. We also define the auxiliary function $\bar a_{\ep}$ and recall its connection with the fractional Laplacians of $\phi$ and $\phi(d(t,x)/\ep)$, as well as its relation to the fractional mean curvature operator.

\subsection{The phase transition $\phi$}

Let $\phi$ be the solution to \eqref{eq:standing wave} and let $H(\xi)$ be the Heaviside function. 

\begin{lem}\label{lem:asymptotics}  
There is a unique solution $\phi \in C^{2, \beta}(\R)$ of \eqref{eq:standing wave}, for some $\beta\in (0,1)$. Moreover, there exists a constant $C>0$ 
and $\vartheta>2s$ (depending only on~$s$) such that
\begin{equation}\label{eq:asymptotics for phi}
\abs{\phi(\xi) - H(\xi) +  \frac{C_{n,s}}{2sW''(0)} \frac{\xi}{|\xi|^{2s+1}}} \leq \frac{C}{\abs{\xi}^{\vartheta}}, \quad \text{for }\abs{\xi} \geq 1,
\end{equation}
with $C_{n,s}$ as in \eqref{eq:Cns}, and
\begin{equation}\label{eq:asymptotics for phi dot}
\frac{1}{C\abs{\xi}^{2s+1}}\leq \dot{\phi}(\xi) \leq \frac{C}{\abs{\xi}^{2s+1}},\quad |\ddot{\phi}(\xi)| \leq \frac{C}{\abs{\xi}^{2s+1}}
 \quad \text{for }\abs{\xi} \geq 1.
\end{equation}
\end{lem}

\begin{proof}
The existence of a unique solution of \eqref{eq:standing wave} is established  in  \cite{CabreSola} for $s=\frac12$, and in 
\cite{CabreSire,PalatucciSavinValdinoci} for any $s\in(0,1)$. 
The estimate \eqref{eq:asymptotics for phi}, as well as the estimate on $\dot\phi$  in \eqref{eq:asymptotics for phi dot}, are proven in \cite{GonzalezMonneau} for  $s=\frac{1}{2}$,  and 
in \cite{DipierroFigalliValdinoci} and \cite{DipierroPalatucciValdinoci}, respectively, when $s\in\left(0,\frac{1}{2}\right)$ and $s\in\left(\frac{1}{2},1\right)$. 
Finally, the estimate on $\ddot\phi$  in \eqref{eq:asymptotics for phi dot} is  established in \cite{MonneauPatrizi2}. 
\end{proof}


\subsection{The auxiliary function $\bar{a}_{\ep}$}\label{aepsubsection}

We now introduce a modified version of the auxiliary function $\bar a_{\ep}$, previously defined in \eqref{oldaep}, based on the approximation obtained in \cite{HasaniPatrizi}.

Let $\Omega_t$ be a bounded domain with smooth boundary $\Gamma_t$, for $t\in[0,T]$.
Throughout this section, let $d=d(t,x)$ denote the smooth extension of the signed distance function $\tilde{d}$ to $\Omega_t$ given in  Definition~\ref{defn:extension}. We also introduce a new parameter $R > 1$, which will be chosen later. 
Define the following auxiliary functions for  $(t,x) \in [0,T] \times \R^n$, 
\begin{equation}\label{b_epsilon}
   \bar  b_\ep[d](t,x): = \int_{\{|z|<R\}} \left[ \phi\left(\frac{d(t,x+ z)}{\ep}\right) - \phi\left( \frac{d(t,x)+ \nabla d(t,x) \cdot z}{\ep} \right) \right] \frac{dz}{\abs{z}^{n+2s}},
\end{equation}
\begin{equation}\label{c_epsilon}
   \bar  c_\ep[d](t,x) :=\frac{1}{\ep^{2s}}\left[ \left(|\nabla d(t,x)|^{2} + \ep^{2+\frac{2s}{1-2s}} \right)^s -1 \right]W' \left( \phi \left( \frac{d(t,x)}{\ep} \right) \right).
\end{equation}
By Lemma~\ref{kernellemma3}, and due to the regularity of $\phi$ and $d$, the integral in \eqref{b_epsilon} is well defined. 
We then define the function $\bar a_\ep= \bar a_\ep[d]( t,x)$,   by
\begin{equation}\label{aepsilondef}
\bar a_{\ep}:= \bar b_\ep+\bar  c_\ep. 
\end{equation}

The following results are proven in \cite{HasaniPatrizi}. The first lemma states that for points $x$ sufficiently close to $\Gamma_t$, $\bar{a}_\ep[d](t,x)$ approximates  the fractional mean curvature of the smooth set $\{d(t,\cdot)>d(t,x)\}$ at the point $x$.
\begin{lem}\cite[Lemma 5.2]{HasaniPatrizi}\label{thm: b_ep fmc} 
For $t\in[0,T]$, let $\Omega_t$ be a bounded domain with smooth boundary. 
Let $d$ be as in Definition \ref{defn:extension} and $\kappa[x,d]$ as in \eqref{kappade}. There exists $\delta_0>0$ such that if  $0<\delta<\delta_0$,   
 and $|d(t,x)|< \delta$, then 
\begin{align*}
\bar{a}_\ep[d](t,x) =\kappa[x,d] + o_\ep(1) + o_\delta(1)+ O(R^{-2s}). 
\end{align*}   
\end{lem}
The next lemma shows that 
$\bar{a}_\ep$ is, up to small errors,  the difference between an $n$-dimensional and a $1$-dimensional fractional Laplacian.
\begin{lem}\cite[Lemma 5.3]{HasaniPatrizi}\label{lem:a_ep and frac laplacians}
For all  $(t,x) \in [0,T] \times \mathbb{R}^n$,
\begin{equation*}
   \mathcal{I}_n^s \left[ \phi \left( \frac{d(t, \cdot)}{\ep} \right) \right](x) - \frac{C_{n,s}}{\ep^{2s}}\mathcal{I}_1^s \phi \left( \frac{d(t, x)}{\ep}\right)= \bar a_\ep[d](t,x) + O(R^{-2s}) + o_\ep(1). 
\end{equation*} 
\end{lem}

Finally, the next result provides estimates for $\bar a_\ep$ and its  derivatives. 

\begin{lem}
\label{lem:a_estimates}
There exist $\tau_\ep=o_\ep(1)$  and a constant $C>0$  such that, if $\tau\geq \tau_\ep$  in Definition \ref{defn:extension}, then for all  $(t,x) \in [0,T] \times \mathbb{R}^n$, 
    \begin{equation*}\label{a_near}
    |\bar a_\ep[d](t,x)| \le  C, 
    \end{equation*}   
    \begin{equation*}\label{partial_a_near}
        |\nabla_x \bar a_\ep[d](t,x)|, |\partial_t \bar a_\ep[d](t,x)|   = \ep^{-1}o_{\ep}(1)R,
    \end{equation*}
    and 
    \begin{equation*}\label{fractional_laplacian_a_ep_near}
    \left|\mathcal{I}_n^s[\bar a_\ep] \right|  =\ep^{-2s}  o_{\ep}(1)R.
\end{equation*}
\end{lem}
\begin{proof}
 The proof follows the arguments of Lemma 5.4 and Corollary 5.5 in \cite{HasaniPatrizi}. 
Although a different approximation of the distance function is used there, the proofs carry over to $d$ defined as in Definition \ref{defn:extension}, upon using estimates \eqref{approxdtau} and selecting $\tau\geq\tau_\ep$, for some suitably chosen $\tau_\ep =o_\ep(1)$. 
\end{proof}

\subsection{Asymptotics of the interaction term}\label{sec:interaction}

The following lemma describes the asymptotic behavior of the nonlocal term $\mathcal{I}_n^s[\phi(\frac{d}{\ep})]$ as $\ep \to 0$, which plays a crucial role in understanding the interaction between multiple fronts (see Subsection \ref{sec:heuristicspart1}). The proof is given in Section \ref{sec:proof of interaction asymptotics}.

\begin{lem}[Asymptotics of the interaction term]\label{lem:interaction asymptotics}
There exists $0<\ell=o_\ep(1)$, such that
for all  $(t,x) \in [0,T] \times \mathbb{R}^n$ satisfying  $|d(t,x)| > \ell$, we have
\begin{equation}\label{eq:interaction limit}
\mathcal{I}_n^s\left[\phi\left(\frac{d(t,\cdot)}{\ep}\right)\right](x) = P[d](t,x) + o_\ep(1),
\end{equation}
where $P[d]$ is defined in \eqref{def:P}. 
\end{lem}
We now construct an approximation of $\mathcal{I}_n^s\left[\phi\left(\frac{d(t,\cdot)}{\ep}\right)\right]$ that remains uniformly bounded up to the boundary $\Gamma_t$. 

    Let $\ell_0>0$, and 
let $g=g[d](t,x)$ be a smooth bounded function  such that
    \begin{align}\label{gdefforPep}
    g[d] = 1 \quad\text{if }  |d| > \ell_0, \qquad g[d]=0\quad \text{if }  |d| < \frac{\ell_0}{2}, \qquad 0\le g \le 1,
\end{align}
and define
\begin{equation}\label{eq:P_ep}
P_\ep[d](t,x) := \mathcal{I}_n^s\left[\phi\left(\frac{d(t,\cdot)}{\ep}\right)\right](x) ~g[d](t,x). 
\end{equation}
The following lemma provides estimates for $P_\ep[d]$ and its derivatives. The proof is given in Section \ref{sec:proof of Pep estimates}.
\begin{lem}\label{lem:Pep_estimates}
There exist $0<\tau_\ep=o_\ep(1)$  and a constant $C>0$ such that if
$\tau\geq \tau_\ep$  in Definition \ref{defn:extension}, then for all   $(t,x) \in [0,T] \times \mathbb{R}^n$
    \begin{equation}\label{Pep_bound}
    \left|P_\ep[d](t,x)\right|\le  C,
    \end{equation}
 \begin{equation}\label{eq:interaction limitbis}
P_\ep[d](t,x) =
P[d](t,x)+ o_\ep(1)\qquad\text{if }|d(t,x)|\geq \ell_0,
\end{equation}
where $P[d]$ is defined in \eqref{def:P}, and
    \begin{equation}\label{Pep_deriv_bound}
        |\nabla_x P_\ep[d](t,x)|, |\partial_t P_\ep[d](t,x)|   = \ep^{-1}o_{\ep}(1).
    \end{equation}
\end{lem}

 \subsection{The correctors $\psi_i$}\label{sec:corrector}
For  $i=1,\ldots,N$  and $t\in[0,T]$, let  $\Omega^i_t$ be  bounded smooth open sets. 
Let $d_i=d_i(t,x)$ denote the smooth extension of the signed distance function $\tilde{d}_i$ to $\Omega^i_t$ as  given in Definition~\ref{defn:extension}.
Assume there exists $\ell_0>0$ such that,
for all $t\in[0,T]$, 
 \begin{equation}\label{Qisigmaemptyimters} \{x\in \R^n\,:\,|d_i(t,x)|<\ell_0\}\cap \{x \in \R^n \,:\,|d_j(t,x)|<\ell_0\}=\emptyset,\qquad i \not=j.
\end{equation}
 Now we introduce two additional small parameters that will be chosen later:
$0<\delta<\min\{1,\ell_0\}$ and $\sigma\in(0,1)$. 
Let $\mu$ be a smooth function such that
\begin{equation}\label{def:mu_definition}\begin{split}
\mu[d_1,\dots,d_N](t,x) =
\begin{cases}
\sigma, & \text{if } |d_i(t,x)| \le \delta \text{ for some } i\in\{1,\ldots,N\}, \\[8pt]
\dfrac{\sigma}{\delta^{2s}}, & \text{if }|d_i(t,x)| \ge 2\delta \text{ for all }  i\in\{1,\ldots,N\},
\end{cases}\\[5pt]
\sigma\leq\mu[d_1,\dots,d_N](t,x)  \leq \dfrac{\sigma}{\delta^{2s}}, \quad \text{otherwise},
\end{split}
\end{equation}
and
\begin{equation}\label{mu_estimates}
    |\partial_t\mu[d_1,\dots,d_N](t,x)|, \, |\nabla_x \mu[d_1,\dots,d_N](t,x)| \le \frac{C}{\delta^{2s+1}}.
\end{equation}
\begin{lem}\label{lem:mu_estimates}
There exists $C>0$ such that, for all  $(t,x) \in [0,T] \times \mathbb{R}^n$,
    \begin{equation*}\left| \mathcal{I}_n^s \left[ \mu[d_1,\dots,d_N](t, \cdot) \right] (x)\right|
    \le \frac{C}{\delta^{4s}}.
\end{equation*}
\end{lem}
\begin{proof}
    The estimate on  $ \mathcal{I}_n^s \left[ \mu[d_1,\dots,d_N](t, \cdot)\right] $ follows from \eqref{def:mu_definition} and \eqref{mu_estimates}. For a proof see   \cite[Lemma 5.6]{HasaniPatrizi}. 
\end{proof}

The linearized operator $\mathcal{L}$ associated to \eqref{eq:standing wave} around $\phi$ is given by
\begin{equation}\label{eq:linearized operator}
\mathcal{L}[\psi] = -C_{n,s}\mathcal{I}^s_1[\psi] + {W}''(\phi) \psi,
\end{equation}
 with $C_{n,s}$ as in \eqref{eq:Cns}.

In the construction of barriers, we will need for each $i \in \{1, \ldots, N\}$ a corrector $\psi_i = \psi_i(\xi; t,x)$ that solves
\begin{equation}\label{eq:linearized wave}
\begin{cases}
\displaystyle{
\mathcal{L}[\psi_i](\xi)
 =   \left( c_0 \dot{\phi}( \xi)
  + \frac{W''\left( \phi (\xi)\right) - W''(0)}{W''(0)} \right)\left(\mu[d_1,\dots,d_N] -\bar{a}_\ep[d_i] -\sum_{j\neq i}P_\ep[d_j]\right),
  } &\xi \in \R\\
\psi_i(\pm\infty;t,x) = 0,
\end{cases}
\end{equation}
where $c_0$ is given by \eqref{def:c0} and the functions $\bar a_\ep$, $P_\ep$ and $\mu$ are defined in \eqref{aepsilondef}, \eqref{eq:P_ep} and \eqref{def:mu_definition}, respectively.

Each corrector $\psi_i$ depends on $(t,x)$ through the signed distance functions $d_1,\dots,d_N$, which appear on the right-hand side of \eqref{eq:linearized wave}. Moreover, although not explicitly indicated, $\psi_i$ also depends on the parameters $\ep$ and $R$ through the functions $\bar a_\ep$ and $P_\ep$, and on $\delta$, $\sigma$ through the function $\mu$.

\begin{lem}\label{lem:psi-reg}
 Assume \eqref{Qisigmaemptyimters}. Let $\tau$  be chosen as in 
 Lemmas~\ref{lem:a_estimates} and \ref{lem:Pep_estimates} in Definition \ref{defn:extension}.
Then, for each $i \in \{1, \ldots, N\}$, there exists a solution $\psi_i = \psi_i(\xi;t,x)\in C_\xi^{1,\beta}(\R)$ to \eqref{eq:linearized wave}
for some $\beta\in(0,1)$. Moreover, there exists $C>0$ such that, for all $0<\ep<1$, $0<\delta<\ell_0$, $\sigma \in (0,1)$, $R>1$, and
 $(\xi, t,x)\in \R\times [0,T] \times\R^n$, the following holds.\\
 If $|d_i(t,x)| < \delta$, then
    \begin{equation}\label{psi_near}
    |\psi_i(\xi;t,x)|,|\dot{\psi}_i(\xi;t,x)|  \le   C,
    \end{equation}
    \begin{equation}\label{partial_psi_near}
        |\nabla_x \psi_i(\xi;t,x)|, |\partial_t \psi_i(\xi;t,x)| = \ep^{-1} o_\ep(1) R.
    \end{equation}
If $|d_i(t,x)| \geq\delta $, then
    \begin{equation}\label{psi_far}
    |\psi_i(\xi;t,x)|, |\dot{\psi}_i(\xi;t,x)| \le 
     \frac{C}{\delta^{2s}(1+|\xi|^{2s})},
    \end{equation}
    and
    \begin{equation}\label{partial_psi_far}
        |\nabla_x \psi_i(\xi;t,x)|, |\partial_t \psi_i(\xi;t,x)| \leq  
            \(  \frac{o_\ep(1) R}{\ep}+\frac{C}{\delta^{2s+1}}\)\frac{1}{1+|\xi|^{2s}}.
    \end{equation}
\end{lem}

\begin{proof}
Under assumptions \eqref{eq:W} on the potential $W$, it is shown in \cite{DipierroFigalliValdinoci}*{Theorem 9.1} that there exists a function $\tilde\psi=\tilde\psi(\xi)\in C^{1,\beta}(\R)$, for some $\beta\in(0,1)$, solving
\[
\begin{cases}
\displaystyle\mathcal{L}[\tilde\psi](\xi) = c_0\,\dot{\phi}(\xi) + \frac{W''(\phi(\xi))-W''(0)}{W''(0)}, & \xi\in\R,\\
\tilde\psi(\pm\infty)=0.
\end{cases}
\]
Moreover, \cite[Lemma 3.2]{MonneauPatrizi2} shows that there exists a constant $C>0$ such that
\begin{equation}\label{tildepsiest}
|\tilde\psi(\xi)|,\,|\dot{\tilde\psi}(\xi)|\le \frac{C}{1+|\xi|^{2s}}\quad\text{for all }\xi\in\R.
\end{equation}
For each $i\in\{1,\ldots,N\}$, we define
\[
\psi_i(\xi;t,x):=\tilde\psi(\xi)\left(\mu[d_1,\dots,d_N](t,x)-\bar{a}_\ep[d_i](t,x)-\sum_{j\neq i}P_\ep[d_j](t,x)\right).
\]
By the linearity of $\mathcal{L}$, it follows that $\psi_i\in C_\xi^{1,\beta}(\R)$ solves \eqref{eq:linearized wave}. The bounds \eqref{psi_near}--\eqref{partial_psi_far} follow from \eqref{tildepsiest}, together with the estimates on  $\bar{a}_\ep$ in Lemma~\ref{lem:a_estimates}, the estimates on $P_\ep$ in Lemma~\ref{lem:Pep_estimates}, the definition of $\mu$ in \eqref{def:mu_definition},  and the derivative bounds \eqref{mu_estimates}.
\end{proof}

We conclude this section with the following estimate for the difference between the $n$- and the $1$-dimensional fractional Laplacians for the function $\psi_i$. 
\begin{lem}\label{lem:ae psi estimate}
Let $\tau$  be chosen as in Lemmas~\ref{lem:a_estimates} and \ref{lem:Pep_estimates} in Definition \ref{defn:extension}, and  assume $\ep/\delta^2=o_\ep(1)$. Then  for all $(t,x) \in [0,T] \times \mathbb{R}^n$, and $i \in \{1, \ldots, N\}$, 
 \begin{align*}
\abs{\ep^{2s} \mathcal{I}_n^s\left[ \psi_i \( \frac{d_i(t,\cdot)}{\ep};t,\cdot\) \right](x) - C_{n,s} \mathcal{I}_1^s[\psi_i\(\cdot;t,x\)]\(\frac{d_i(t,x)}{\ep}\)}= Ro_\ep(1).
\end{align*}
\end{lem}
\begin{proof}
The proof follows the same argument as in \cite[Lemma 5.3]{HasaniPatrizi}, using the estimates on $\psi_i$ provided by Lemma \ref{lem:psi-reg}. 
\end{proof}

\section{Constructions of barriers}\label{barriersection}
In this section, we construct barriers, namely subsolutions and supersolutions to \eqref{eq:pde} and  \eqref{initial_data}. We focus on the construction of subsolutions, as supersolutions can be obtained in a similar way.
 \subsection{Subsolutions to \eqref{eq:pde}}
 
Let $\Omega_t^i$, $i=1,\ldots N$, $t\in [0,T]$,  be a family of smooth bounded domains satisfying the nesting condition \eqref{nestingassumption}.
Let $\tilde{d}_i(t,x)$ be the signed distance function associated to the set  $ \Omega^i_t $, then  $\Gamma_t^i =  \{ x \in \R^n : \tilde{d}_i(t,x) = 0\}$.
Assume  that there exists  a $\rho>0$ such that, for all $1 \leq i \leq N$, $\tilde{d}_i(t,x)$ is smooth in the set
\begin{equation*}\label{eq:Q2rho}
Q_{3\rho}^i: = \{ (t,x) \in[0,T]\times\R^n: |\tilde{d}_i(t,x)| < 3\rho\},
\end{equation*}
and let $d_i$ be the smooth extension of $\tilde{d}_i$ given  in Definition \ref{defn:extension}, with $\tau$ chosen as in Lemmas~\ref{lem:a_estimates} and \ref{lem:Pep_estimates}. Assume in addition that there exists $\sigma>0$ such that
\begin{equation}\label{eq:fmc for d}
\partial_t d_i
\leq
c_0\!\Bigg(
\kappa\!\big[x, d_i(t,\cdot)\big]
+
\sum_{j\neq i}
P[d_j]
-
\sigma
\Bigg)
\quad \text{in } Q_{2\tilde\sigma},
\end{equation}
 where for $\kappa[x,d]$ is the fractional mean curvature operator defined in \eqref{kappade}, $P[d]$ is defined by  \eqref{PepsimP_heur}, and $\tilde\sigma>0$ is defined by 
 \begin{equation}\label{alphasigma}\alpha:= W''(0),\quad \tilde{\sigma}:=\frac{\sigma}{\alpha}.\end{equation}
By \eqref{nestingassumption}, there exists $\ell_0>0$ such that condition \eqref{Qisigmaemptyimters} holds  for all $t\in[0,T]$. 
Possibly decreasing $\sigma$, we may assume  $2\tilde\sigma<\min\{\rho,\ell_0,1\}$. We define the smooth barrier $v^{\ep}(t,x)$ by
\begin{equation}\label{eq:barrier defn}
\begin{split}
v^{\ep}(t,x) =&\sum_{i=1}^N \phi\left( \frac{d_i(t,x)- \tilde{\sigma}^N }{\ep}\right) +\ep^{2s}\sum_{i=1}^N  \psi_i \left( \frac{d_i(t,x)-\tilde{\sigma}^N}{\ep};t,x\right)
\\
&+ \frac{\ep^{2s}}{\alpha} \sum_{i=1}^N \bar{a}_\ep\left[d_i - \tilde{\sigma}^N \right](t,x) - \frac{\ep^{2s}}{\alpha}\,\mu[d_1-\tilde{\sigma}^N,\dots,d_N-\tilde{\sigma}^N](t,x),
\end{split}
\end{equation}
where $\phi$ is the solution of \eqref{eq:standing wave},  $\psi_i$ and $\mu$  solve \eqref{eq:linearized wave} and \eqref{def:mu_definition} respectively  with the signed distance functions $d_1-\tilde\sigma^N,\dots,d_N-\tilde\sigma^N$. 
Recall that  $\psi_i$ also depends on $\ep$, $R$, $\delta$, and $\sigma$ through the functions $\bar a_\ep$, $P_\ep$  and $\mu$. The parameter $\sigma>0$ is chosen as in \eqref{eq:fmc for d}, and we assume the following condition on $\delta$:
\begin{equation}\label{delta:def} \delta=o_\ep(1),\quad\frac{\ep}{\delta^2}=o_\ep(1).
\end{equation}
\begin{lem}[Subsolutions to \eqref{eq:pde}] \label{lem:barrier}
 Assume \eqref{eq:fmc for d} with $c_0$ as in \eqref{def:c0}, and \eqref{Qisigmaemptyimters}. Let $v^{\ep}$ be defined as in \eqref{eq:barrier defn}
 with $0<2\tilde\sigma<\min\{\rho,\ell_0,1\}$, $R >1$, and $\delta=\delta(\ep)$ satisfying \eqref{delta:def}.
  Then there exists $R_0=R_0(\sigma)$ and $\ep_0=\ep_0(\sigma)>0$ such that for all $R>R_0$ and $0<\ep<\ep_0$, $v^{\ep}$ satisfies 
\begin{equation}\label{eq:pde sub}
\ep \partial_t v^{\ep} - \mathcal{I}_n^s[v^{\ep} ] +\frac{1}{\ep^{2s}} W'(v^{\ep} ) \leq -\frac{\sigma}{4} \quad \hbox{in}~(0,T) \times \R^n.
\end{equation}
\end{lem}

\begin{proof}
For convenience, we use the following notation throughout the proof:
\begin{equation}\label{eq:notation}
\begin{aligned}
\phi_i &:= \phi \( \frac{d_i(t,x) - \tilde{\sigma}^N}{\ep}\)\\
 \psi_i &:= \psi_i \( \frac{d_i(t,x) - \tilde{\sigma}^N}{\ep};t,x\)\\
 \tilde{\phi}_i &:= \phi \( \frac{d_i(t,x) - \tilde{\sigma}^N}{\ep}\)- H \( \frac{d_i(t,x) - \tilde{\sigma}^N}{\ep}\)\\
\bar{a}_{\ep}^i
	&:= \bar{a}_{\ep}[d_i- \tilde{\sigma}^N](t,x)\\
\mu &:= \mu[d_1-\tilde{\sigma}^N,\dots,d_N-\tilde{\sigma}^N](t,x)\\
P_\ep^i & := P_\ep[d_i-\tilde{\sigma}^N](t,x)
\end{aligned}
\end{equation}
We note that it will be important for the reader to remember the dependence of $\psi_i$ on the variables $t,x$ and $(d_i(t,x) - \tilde{\sigma}^N)/\ep$ when taking derivatives in $t,x$. With this notation, the barrier \eqref{eq:barrier defn} can be written as
$$v^{\ep}(t,x) = \sum_{i=1}^N \phi_i + \ep^{2s}\sum_{i=1}^N \psi_i + \frac{\ep^{2s}}{\alpha} \sum_{i=1}^N \bar{a}_\ep^i - \frac{\ep^{2s}}{\alpha}\mu.$$
We begin by computing the time derivative of $v^\ep$ at $(t,x)$:
\begin{align*}
\ep \partial_t v^{\ep}(t,x)
	&=  \sum_{i=1}^N  \dot{\phi}_i \, \partial_td_i(t,x)
	 + \ep^{2s} \sum_{i=1}^N \left[\dot\psi_i \partial_t d_i(t,x)
	 +\ep \partial_t\psi_i\right] + \frac{\ep^{2s+1}}{\alpha} \sum_{i=1}^N \partial_t \bar{a}_\ep^i-\frac{\ep^{2s+1}}{\alpha}\partial_t\mu.
\end{align*}
By Lemmas \ref{lem:a_estimates} and \ref{lem:psi-reg}, and the derivative estimate \eqref{mu_estimates}, we have that for each $i=1, \ldots, N$,
\begin{align*}
 \ep^{2s} \dot{\psi_i} \partial_td_i(t,x) + \ep^{2s+1} \partial_t \psi_i + \frac{\ep^{2s+1}}{\alpha} \partial_t \bar{a}_\ep^i-\frac{\ep^{2s+1}}{\alpha}\partial_t\mu &= O\(\frac{\ep^{2s}}{\delta^{2s}}\)+
 \ep^{2s}Ro_\ep(1)+O\(\frac{\ep^{2s+1}}{\delta^{2s+1}}\)\\&
 =Ro_\ep(1),
\end{align*}
where we used that  $\ep/\delta^2=o_\ep(1)$ in the last equality. Therefore,
we have
\begin{equation}\label{eq:time_derivative_of_v_ep'}
\ep \partial_t v^\ep (t,x)= \sum_{i=1}^N \dot{\phi}_i \partial_t d_i (t,x)+Ro_\ep(1).
\end{equation}
Next, we consider the nonlocal term. We compute
\begin{align*}
\mathcal{I}_n^s[v^\ep(t,\cdot)](x) =& \sum_{i=1}^N \left( \mathcal{I}_n^s\left[\phi_i  \right] + \ep^{2s} \mathcal{I}_n^s \left[ \psi_i\right] + \frac{\ep^{2s}}{\alpha} \mathcal{I}_n^s[\bar{a}_\ep^i](x) \right)-\frac{\ep^{2s}}{\alpha}\mathcal{I}_n^s[\mu].
\end{align*}
Fix an index $i_0 \in \{1,\ldots,N\}$.
Using that $\phi$ satisfies \eqref{eq:standing wave} and applying Lemma~\ref{lem:a_ep and frac laplacians}, we obtain
\begin{align*}
\mathcal{I}_n^s\left[\phi_{i_0} \right]&=  \mathcal{I}_n^s\left[\phi_{i_0}\right]
 - \frac{C_{n,s}}{\ep^{2s}} \mathcal{I}_1^s[\phi_{i_0}]\left( \frac{d_{i_0}(t,x) -\tilde{\sigma}^N}{\ep} \right) +\frac{C_{n,s}}{\ep^{2s}} \mathcal{I}_1^s[\phi_{i_0}]\left( \frac{d_{i_0}(t,x) -\tilde{\sigma}^N}{\ep} \right) \\
& = \bar a_\ep^{i_0}  + \frac{1}{\ep^{2s}}W'(\phi_{i_0})+ O(R^{-2s}) +  o_\ep(1).
\end{align*}
Assume
\begin{equation}\label{d_iconditions_lemmasub}
|d_i(t,x)-\tilde\sigma^N|\geq\ell_0\qquad\text{for }i\neq i_0.
\end{equation}
Then  by the definition of $ P^i_\ep$ in  \eqref{eq:P_ep}, we have
\begin{equation}\label{d_iconditions_lemmasub_bis}
 \mathcal{I}_n^s[\phi_i](x) = P^i_\ep\qquad\text{for }i\neq i_0.
 \end{equation}
 Combining this with the above identity, we obtain
\begin{align*}
\sum_{i=1}^N  \mathcal{I}_n^s\left[\phi_i  \right](x) = \bar a_\ep^{i_0} +\sum_{i\neq i_0}P^i_\ep + \frac{1}{\ep^{2s}}W'(\phi_{i_0})+ O(R^{-2s}) +  o_\ep(1).
\end{align*}
 Next, recalling \eqref{eq:linearized operator}, that
$\alpha = W''(0)$, and using that $\psi_i$ solves \eqref{eq:linearized wave}, we compute
\begin{align*}
\ep^{2s}\mathcal{I}_n^s[\psi_i] =&\ep^{2s}\mathcal{I}_n^s[\psi_i] - C_{n,s}\mathcal{I}^s_1[\psi_i]\left( \frac{d_i(t,x) - \tilde{\sigma}^N}{\ep} \right) + {W}''(\phi_i) \psi_i - \mathcal{L}[\psi_i]\left( \frac{d_i(t,x) - \tilde{\sigma}^N}{\ep} \right)\\
=& \ep^{2s}\mathcal{I}_n^s[\psi_i](x) - C_{n,s}\mathcal{I}^s_1[\psi_i]\left( \frac{d_i(t,x) - \tilde{\sigma}^N}{\ep} \right) + {W}''(\phi_i) \psi_i
\\&- \left(c_0 \dot{\phi_i} + \frac{W''\left( \phi_i \right) }{\alpha} -1\right)  \left( \mu-\bar{a}_{\ep}^i - \sum_{j\neq i}P^j_\ep \right).
\end{align*}
Since $\ep/\delta^2=o_\ep(1)$, we may apply Lemma \ref{lem:ae psi estimate} and thereby obtain
\begin{align*}
\ep^{2s}\mathcal{I}_n^s[\psi_i](x) =  {W}''(\phi_i) \psi_i -  \left(c_0 \dot{\phi_i} + \frac{W''\left( \phi_i \right)}{\alpha}-1 \right)  \left(\mu-\bar{a}_{\ep}^i - \sum_{j\neq i}P^j_\ep\right)+ Ro_\ep(1).
\end{align*}
Moreover, by Lemma \ref{lem:a_estimates},
\begin{align*}
\ep^{2s}\mathcal{I}_n^s[\bar{a}_\ep^i](x) = Ro_\ep(1).
\end{align*}
Finally, by Lemma \ref{lem:mu_estimates} and the fact that $\ep/\delta^2=o_\ep(1)$,
\begin{align*}
\ep^{2s}\mathcal{I}_n^s[\mu](x) = O\(\frac{\ep^{2s}}{\delta^{4s}}\)=o_\ep(1).
\end{align*}
Therefore, the fractional Laplacian of $v^\ep$ can be written as
\begin{equation*}
\begin{split}
 \mathcal{I}_n^s[v^\ep(t,\cdot)](x) &= \bar a_\ep^{i_0} +\sum_{i\neq i_0}P^i_\ep + \frac{1}{\ep^{2s}}W'(\phi_{i_0})
 \\&\quad+   {W}''(\phi_{i_0}) \psi_{i_0}-  \left(c_0 \dot{\phi}_{i_0} + \frac{W''\left( \phi_{i_0} \right)}{\alpha} -1\right)  \left(\mu-\bar{a}_{\ep}^{i_0} - \sum_{i\neq i_0}P^i_\ep\right)
\\&\quad+ \sum_{i\neq i_0}\left\{   {W}''(\phi_i) \psi_i -  \left(c_0 \dot{\phi}_i + \frac{W''\left( \phi_i \right)}{\alpha} -1\right)  \left(\mu-\bar{a}_{\ep}^i - \sum_{j\neq i}P^j_\ep\right) \right\} \\& \quad+  O(R^{-2s}) +Ro_\ep(1).
\end{split}
\end{equation*}
Simplifying the terms $ \bar a_\ep^{i_0}$ and  $\sum_{i\neq i_0}P_\ep^i$, we obtain
\begin{equation}\label{eq:fractional_Laplacian_of_v_ep'}
\begin{split}
 \mathcal{I}_n^s[v^\ep(t,\cdot)](x) &=\frac{1}{\ep^{2s}}W'(\phi_{i_0})+   {W}''(\phi_{i_0}) \psi_{i_0}-  \left(c_0 \dot{\phi}_{i_0} + \frac{W''\left( \phi_{i_0} \right)}{\alpha} \right)  \left(\mu-\bar{a}_{\ep}^{i_0} - \sum_{i\neq i_0}P^i_\ep\right)+\mu
\\&\quad+ \sum_{i\neq i_0}\left\{   {W}''(\phi_i) \psi_i -  \left(c_0 \dot{\phi}_i + \frac{W''\left( \phi_i \right)}{\alpha} -1\right)  \left(\mu-\bar{a}_{\ep}^i - \sum_{j\neq i}P^j_\ep\right) \right\} \\& \quad+  O(R^{-2s}) +Ro_\ep(1).
\end{split}
\end{equation}
Next, we compute $W'(v^\ep(t,x))$.
Performing a Taylor expansion of $W'$ around $\phi_{i_0}$ and using the periodicity of $W$, we obtain
\begin{equation}\label{eq:Wprimeforbarrier'}
\begin{aligned}
&W'\left(\sum_{i=1}^N \phi_i + \ep^{2s}  \sum_{i=1}^N \psi_i + \frac{\ep^{2s}}{\alpha} \sum_{i=1}^N \bar{a}_\ep^i - \frac{\ep^{2s}}{\alpha}\mu\right)
\\& =W'\left(\phi_{i_0}+\sum_{i\neq i_0}^N \tilde\phi_i + \ep^{2s}  \sum_{i=1}^N \psi_i + \frac{\ep^{2s}}{\alpha} \sum_{i=1}^N \bar{a}_\ep^i - \frac{\ep^{2s}}{\alpha}\mu\right)\\
	& =W'(\phi_{i_0}) + W''(\phi_{i_0})\(\sum_{i\neq i_0}^N \tilde\phi_i + \ep^{2s}  \sum_{i=1}^N \psi_i + \frac{\ep^{2s}}{\alpha} \sum_{i=1}^N \bar{a}_\ep^i - \frac{\ep^{2s}}{\alpha}\mu\) \\
	&\quad+ O\(\(\sum_{i\neq i_0}^N \tilde\phi_i + \ep^{2s}  \sum_{i=1}^N \psi_i + \frac{\ep^{2s}}{\alpha} \sum_{i=1}^N \bar{a}_\ep^i - \frac{\ep^{2s}}{\alpha}\mu\)^2\).
\end{aligned}
\end{equation}
Using the estimates for $\bar a_\ep$ and $\psi_i$ provided by
Lemmas~\ref{lem:a_estimates} and \ref{lem:psi-reg}, and recalling that $0\le \mu\le \sigma/\delta^{2s}$ by \eqref{def:mu_definition}, we deduce
\begin{align*}
 \frac{1}{\ep^{2s}} O&\(\(\sum_{i\neq i_0}^N \tilde\phi_i + \ep^{2s}  \sum_{i=1}^N \psi_i + \frac{\ep^{2s}}{\alpha} \sum_{i=1}^N \bar{a}_\ep^i - \frac{\ep^{2s}}{\alpha}\mu\)^2\)=   \frac{1}{\ep^{2s}}\sum_{i\not= i_0} O\(\tilde{\phi}_i\)^2
		 + O\(\frac{\ep^{2s}}{\delta^{4s}}\)+o_\ep(1)\\
&=\frac{1}{\ep^{2s}}\sum_{i\not= i_0} O\(\tilde{\phi}_i\)^2+o_\ep(1),
\end{align*}
where we used again that $\ep/\delta^2=o_\ep(1)$.
 Combining this last estimate with  \eqref{eq:time_derivative_of_v_ep'},  \eqref{eq:fractional_Laplacian_of_v_ep'} and \eqref{eq:Wprimeforbarrier'}, we obtain
\begin{align*}
\mathcal{J}[v^\ep]
:= &\ep \partial_t v^{\ep}(t,x) - \mathcal{I}^s_n [v^{\ep}(t,\cdot)](x)  +\frac{1}{\ep^{2s}}  W'(v^\ep(t,x))\\
= & \sum_{i=1}^N\dot{\phi}_i \partial_td_i(t,x) \\
&-\frac{1}{\ep^{2s}}W'(\phi_{i_0})-   {W}''(\phi_{i_0}) \psi_{i_0}+  \left(c_0 \dot{\phi}_{i_0} + \frac{W''\left( \phi_{i_0} \right)}{\alpha} \right)  \left(\mu-\bar{a}_{\ep}^{i_0} - \sum_{i\neq i_0}P^i_\ep\right)-\mu
\\&- \sum_{i\neq i_0}\left\{   {W}''(\phi_i) \psi_i -  \left(c_0 \dot{\phi}_i + \frac{W''\left( \phi_i \right)}{\alpha} -1\right)  \left(\mu-\bar{a}_{\ep}^i - \sum_{j\neq i}P^j_\ep\right) \right\} \\
	&+\frac{W'(\phi_{i_0})}{\ep^{2s}} + \frac{W''(\phi_{i_0})}{\ep^{2s}}\(\sum_{i\neq i_0}^N \tilde\phi_i + \ep^{2s}  \sum_{i=1}^N \psi_i + \frac{\ep^{2s}}{\alpha} \sum_{i=1}^N \bar{a}_\ep^i - \frac{\ep^{2s}}{\alpha}\mu\) \\
	&+ Ro_\ep(1)+  O(R^{-2s})+ \frac{1}{\ep^{2s}}\sum_{i\not= i_0} O\(\tilde{\phi}_i\)^2.
	\end{align*}
Grouping  and simplifying the terms $\frac{1}{\ep^{2s}}W'(\phi_{i_0})$, $ {W}''(\phi_{i_0}) \psi_{i_0}$, $ \frac{{W}''(\phi_{i_0}) }{\alpha} \bar{a}_{\ep}^{i_0}$, and
 $ \frac{{W}''(\phi_{i_0}) }{\alpha} \mu$, we get
\begin{equation}\label{Jvep_almostfinal}
\begin{aligned}
\mathcal{J}[v^\ep]=&
 \sum_{i\neq i_0}^N\dot{\phi}_i \partial_td_i(t,x) \\
& + \dot{\phi}_{i_0}  \left\{ \partial_td_{i_0}(t,x)-c_0\left(\bar{a}_{\ep}^{i_0} + \sum_{i\neq i_0}P^i_\ep-\mu\right)\right\}-\mu
\\&+ \sum_{i\neq i_0}\left\{  (W''(\phi_{i_0})- {W}''(\phi_i)) \psi_i +  \left(c_0 \dot{\phi}_i + \frac{W''\left( \phi_i \right)}{\alpha} -1\right)  \left(\mu-\bar{a}_{\ep}^i - \sum_{j\neq i}P^j_\ep\right) \right\} \\
	& + \frac{W''(\phi_{i_0})}{\alpha}\left(\sum_{i\neq i_0}^N \frac{\alpha\tilde\phi_i }{\ep^{2s}}+  \sum_{i\neq i_0 }^N \bar{a}_\ep^i -
\sum_{i\neq i_0}P^i_\ep\right)  + Ro_\ep(1)+  O(R^{-2s})+ \frac{1}{\ep^{2s}}\sum_{i\not= i_0} O\(\tilde{\phi}_i\)^2.
\end{aligned}
\end{equation}
Now, by Lemma \ref{lem:a_ep and frac laplacians},  \eqref{d_iconditions_lemmasub_bis}, and the fact that $\phi$ satisfies \eqref{eq:standing wave}, we obtain
\begin{align*}
\sum_{i\neq i_0 }^N \bar{a}_\ep^i &
= \sum_{i\neq i_0 }^N  \mathcal{I}_n^s \left[ \phi \left( \frac{d_i(t, \cdot)}{\ep} \right) \right](x) - \frac{C_{n,s}}{\ep^{2s}} \sum_{i\neq i_0 }^N\mathcal{I}_1^s \phi \left( \frac{d_i(t, x)}{\ep}\right) + O(R^{-2s}) + o_\ep(1)\\
&= \sum_{i\neq i_0}P^i_\ep- \frac{1}{\ep^{2s}} \sum_{i\neq i_0 }^N W'\(\phi_i \)+ O(R^{-2s}) + o_\ep(1).
\end{align*}
Expanding $W'$ around $0$ and using its periodicity, $W'(0)=0$, and $\alpha=W''(0)$, we further have
\begin{align*}
W'\(\phi_i \)=W'(\tilde \phi_i )=\alpha \tilde\phi_i +O\( \tilde\phi_i\)^2. 
\end{align*}
These two  estimates yields
\begin{align*}
\sum_{i\neq i_0}^N \frac{\alpha\tilde\phi_i }{\ep^{2s}}+  \sum_{i\neq i_0 }^N \bar{a}_\ep^i -
\sum_{i\neq i_0}P^i_\ep= \frac{1}{\ep^{2s}} \sum_{i\neq i_0}^NO\( \tilde\phi_i\)^2+ O(R^{-2s}) + o_\ep(1).
\end{align*}
Inserting this expression into \eqref{Jvep_almostfinal}, we finally obtain
\begin{equation}\label{Jvep_final}
\begin{aligned}
\mathcal{J}[v^\ep]=&
 \sum_{i\neq i_0}^N\dot{\phi}_i \partial_td_i(t,x) \\
& + \dot{\phi}_{i_0}  \left\{ \partial_td_{i_0}(t,x)-c_0\left(\bar{a}_{\ep}^{i_0} + \sum_{i\neq i_0}P^i_\ep-\mu\right)\right\}-\mu
\\&+ \sum_{i\neq i_0}\left\{  (W''(\phi_{i_0})- {W}''(\phi_i)) \psi_i +  \left(c_0 \dot{\phi}_i + \frac{W''\left( \phi_i \right)}{\alpha} -1\right)  \left(\mu-\bar{a}_{\ep}^i - \sum_{j\neq i}P^j_\ep\right) \right\} \\&
 + Ro_\ep(1)+  O(R^{-2s})+ \frac{1}{\ep^{2s}}\sum_{i\not= i_0} O\(\tilde{\phi}_i\)^2.
\end{aligned}
\end{equation}
We now  consider three cases: 1) there exists $i_0$ such that  $|d_{i_0}(t,x) - \tilde\sigma^N| < \delta$; 2) there exists $i_0$ such that   $\delta\leq |d_{i_0}(t,x) - \tilde\sigma^N| <\ell_0$; 
3) $|d_i(t,x) - \tilde\sigma^N| \geq \ell_0$
for all $i$.

\medskip
\noindent
\textbf{Case 1:} $|d_{i_0}(t,x) - \tilde\sigma^N| < \delta$.

Since $\tilde\sigma<\ell_0/2$,  $\tilde\sigma<1$, and $\delta=o_\ep(1)$,  from \eqref{Qisigmaemptyimters} it follows that, 
for $\ep$ sufficiently small, condition \eqref{d_iconditions_lemmasub} holds, and
 \begin{equation}\label{Pi0-Pi0deltaest}
|d_{i}(t,x)|\ge \ell_0,\qquad 
\operatorname{sgn}(d_i(t,x) - \tilde{\sigma}^N)=\operatorname{sgn}(d_i(t,x) )\quad\text{for }i \neq i_0.  
 \end{equation}
   Therefore, using  \eqref{eq:asymptotics for phi},  \eqref{eq:asymptotics for phi dot} and \eqref{psi_far},  we obtain
\begin{equation*}
\tilde\phi_i = O\(\frac{\ep^{2s}}{ \ell_0^{2s}}\), \quad\dot\phi_i = O\(\frac{\ep^{2s+1}}{ \ell_0^{2s+1}}\),\quad \psi_i = O\(\frac{\ep^{2s}}{ \ell_0^{2s}}\)  \quad \text{ for } i \neq i_0.
\end{equation*}
Consequently, by the uniform bounds on  $\bar a_\ep^i$ and $P^j_\ep$ provided by Lemmas~\ref{lem:a_estimates} and \ref{lem:Pep_estimates}, and using
the identity $\mu=\sigma$ from  \eqref{def:mu_definition},  we deduce 
$$\sum_{i\neq i_0}^N\dot{\phi}_i \partial_td_i(t,x),\,  \sum_{i\neq i_0}^N c_0 \dot{\phi}_i \left(\mu-\bar{a}_{\ep}^i - \sum_{j\neq i}P^j_\ep\right)=o_\ep(1),$$
$$\sum_{i\neq i_0} (W''(\phi_{i_0})- {W}''(\phi_i)) \psi_i =O\(\frac{\ep^{2s}}{ \ell_0^{2s}}\)=o_\ep(1),$$
\begin{align*}
\sum_{i\neq i_0} \left(\frac{W''\left( \phi_i \right)}{\alpha} -1\right)  \left(\mu-\bar{a}_{\ep}^i - \sum_{j\neq i}P^j_\ep\right)&=\sum_{i\neq i_0} \frac{W''\left( \phi_i \right)-W''(0)}{\alpha}   \left(\mu-\bar{a}_{\ep}^i - \sum_{j\neq i}P^j_\ep\right)\\
&=\sum_{i\neq i_0} O(\tilde\phi_i )= O\(\frac{\ep^{2s}}{ \ell_0^{2s}\delta^{2s}}\)
=o_\ep(1),
\end{align*}
and
$$ \frac{1}{\ep^{2s}}\sum_{i\not= i_0} O\(\tilde{\phi}_i\)^2=O\(\frac{\ep^{2s}}{ \ell_0^{4s}}\)=o_\ep(1). $$
Combining these estimates with \eqref{Jvep_final}, and using that $\mu=\sigma$, we conclude
\begin{equation}\label{Jvep_final_case1}
\begin{aligned}
\mathcal{J}[v^\ep]=&\dot{\phi}_{i_0}  \left\{ \partial_td_{i_0}(t,x)-c_0\left(\bar{a}_{\ep}^{i_0} + \sum_{i\neq i_0}P^i_\ep-\sigma\right)\right\}-\sigma +Ro_\ep(1)+  O(R^{-2s}).
\end{aligned}
\end{equation}
 Next, recall the definition of $P[d]$ in \eqref{PepsimP_heur}. By  \eqref{Pi0-Pi0deltaest},
\begin{align*}
\left|\sum_{i\neq i_0}P[d_i](t,x) -\sum_{i\neq i_0}P[d_i-\tilde\sigma^N](t,x) \right|\leq \sum_{i\neq i_0}\int_{\{|d_i(t,x+z)|\leq \tilde\sigma^N\}}\frac{dz}{|z|^{n+2s}}.
\end{align*}
Now, if $|d_i(t,x+z)| \le \tilde\sigma^N$ and $|d_i(t,x)| \ge \ell_0$, then necessarily $|z| \geq c$, where $c$ depends on $\ell_0$ and the Lipschitz norm of $d_i$. Moreover, since the zero level set of $d_i$ is smooth, we have
$|\{z:|d(x+z)|\leq\tilde\sigma^N\}|\le C\tilde\sigma^N$.
Hence,
\begin{align}\label{Plipschitz}
\sum_{i\neq i_0}P[d_i](t,x) -\sum_{i\neq i_0}P[d_i-\tilde\sigma^N](t,x)=O(\tilde\sigma^N).
\end{align}
By the above equality, \eqref{eq:interaction limitbis} applied with the argument $d_i - \tilde\sigma^N$, and  \eqref{d_iconditions_lemmasub},  we then have
$$\sum_{i\neq i_0}P^i_\ep =\sum_{i\neq i_0}P[d_i]+o_\ep(1)+O(\tilde\sigma^N).$$
By Lemma \ref{thm: b_ep fmc} and the fact that $\delta=o_\ep(1)$, we have 
$$\bar{a}_\ep^{i_0}=\kappa[x, d_{i_0}(t,\cdot) - \tilde\sigma^N] + o_\ep(1)+ O(R^{-2s})=\kappa[x, d_{i_0}(t,\cdot) ]+ o_\ep(1)+ O(R^{-2s}).$$
Combining these estimates with \eqref{eq:fmc for d}, and using that $\dot\phi_{i_0} \geq 0$, we obtain
\begin{align*}
\dot{\phi}_{i_0}&\left[\partial_t d_{i_0}(t,x) - c_0\left(\bar{a}_\ep^{i_0} + \sum_{i\neq i_0}P^i_\ep - \sigma\right)\right] \\&\leq \left[c_0\Big( \kappa[x, d_{i_0}(t,\cdot) ]+\sum_{i\neq i_0}P[d_i]-\sigma\Big)- c_0\left(\bar{a}_\ep^{i_0} + \sum_{i\neq i_0}P^i_\ep - \sigma\right)\right] \\&
=O(\sigma^N)+o_\ep(1)+ O(R^{-2s}).
\end{align*}
Substituting the above estimate into  \eqref{Jvep_final_case1}, we deduce that
\[
\mathcal{J}[v^\ep] \leq -\sigma + O(\sigma^N)  + Ro_\ep(1)+ O(R^{-2s}).
\]
 If $N\geq2$, we may choose $\sigma$ sufficiently small, so that $O(\sigma^N) \le \sigma/2$. Next, choosing $R_0 = R_0(\sigma)$ sufficiently large so that for all $R > R_0$ we have $|O(R^{-2s})| \leq \sigma/8$, and  then selecting $\ep_0 = \ep_0(R_0, \sigma) = \ep_0(\sigma)$ small enough so that for all $0 < \ep < \ep_0$, we have $|Ro_\ep(1)| \leq \sigma/8$, we obtain
\[
\mathcal{J}[v^\ep] \leq -\frac{\sigma}{4}.
\]

If $N=1$, then $\sum_{i\neq i_0}P^i_\ep=0$, and the term $O(\sigma^N)$ does not appear. Therefore, the same inequality follows by choosing $R$ sufficiently large and $\ep$ sufficiently small.

This proves \eqref{eq:pde sub} for Case 1.

\medskip
\noindent
\textbf{Case 2:} $\delta\leq |d_{i_0}(t,x) - \tilde\sigma^N| <\ell_0$.

Since $|d_{i_0}(t,x) - \tilde\sigma^N| <\ell_0$,  \eqref{Qisigmaemptyimters} implies that condition \eqref{d_iconditions_lemmasub} holds. 
As in Case 1, and using that $\ep/\delta^2 = o_\ep(1)$ together with the bound $0\leq \mu\le \sigma/\delta^{2s}$ from  \eqref{def:mu_definition}, we have
\begin{equation}\label{estimateCase2_mainlemma}
\begin{gathered}
\sum_{i\neq i_0}^N\dot{\phi}_i \partial_td_i(t,x),\,  \sum_{i\neq i_0}^N c_0 \dot{\phi}_i \left(\mu-\bar{a}_{\ep}^i - \sum_{j\neq i}P^j_\ep\right)=O\(\frac{\ep^{2s+1}}{\delta^{4s+1}}\)=o_\ep(1),\\
\sum_{i\neq i_0} (W''(\phi_{i_0})- {W}''(\phi_i)) \psi_i =O\(\frac{\ep^{2s}}{\delta^{2s}}\)=o_\ep(1),\\
\sum_{i\neq i_0} \left(\frac{W''\left( \phi_i \right)}{\alpha} -1\right)  \left(\mu-\bar{a}_{\ep}^i - \sum_{j\neq i}P^j_\ep\right)=\sum_{i\neq i_0} O(\tilde\phi_i )\(1+\frac{1}{\delta^{2s}}\)= O\(\frac{\ep^{2s}}{\delta^{4s}}\)
=o_\ep(1),\\
 \frac{1}{\ep^{2s}}\sum_{i\not= i_0} O\(\tilde{\phi}_i\)^2=O\(\frac{\ep^{2s}}{\delta^{4s}}\)=o_\ep(1). 
\end{gathered}
\end{equation}
Moreover, since $|d_{i_0}(t,x) - \tilde\sigma^N| \geq \delta$, by estimate \eqref{eq:asymptotics for phi dot} we have $ \dot{\phi}_{i_0}  =O\(\frac{\ep^{2s+1}}{\delta^{2s+1}}\)$. Using this  together with 
the uniform bounds on $\bar a_\ep^i$ and $P^j_\ep$ provided by Lemmas~\ref{lem:a_estimates} and \ref{lem:Pep_estimates}, and the estimate 
$0\leq \mu\le \sigma/\delta^{2s}$, we obtain
$$ \dot{\phi}_{i_0}  \left[ \partial_td_{i_0}(t,x)-c_0\left(\bar{a}_{\ep}^{i_0} + \sum_{i\neq i_0}P^i_\ep-\mu\right)\right]=O\(\frac{\ep^{2s+1}}{\delta^{4s+1}}\)=o_\ep(1).$$
Combining these estimates with  \eqref{Jvep_final} and using that  $\mu\ge \sigma$ by definition \eqref{def:mu_definition},  yields 
\[
\mathcal{J}[v^\ep] \leq -\mu + O(R^{-2s}) + Ro_\ep(1)\le -\sigma + O(R^{-2s}) + Ro_\ep(1).
\]
Arguing as in Case 1, \eqref{eq:pde sub} follows.

\medskip
\noindent
\textbf{Case 3:} $|d_i(t,x) - \tilde\sigma^N| \geq \ell_0$
for all $i$.

In this case, we fix an arbitrary index $i_0$. Since condition \eqref{d_iconditions_lemmasub} is clearly satisfied, the same argument as in Case~2 yields the estimates in \eqref{estimateCase2_mainlemma}. 
Moreover, since $|d_{i_0}(t,x) - \tilde\sigma^N| \geq \ell_0$, by estimate \eqref{eq:asymptotics for phi dot} we have $ \dot{\phi}_{i_0}  =O\(\frac{\ep^{2s+1}}{\ell_0^{2s+1}}\)$.

Arguing as in Case 2, \eqref{eq:pde sub} follows.
\end{proof}


\subsection{Subsolutions to \eqref{initial_data}}

\begin{lem}[Subsolutions to \eqref{initial_data}]\label{lem:initial_sub}
Assume \eqref{Omega_0^iassumptions}. Let $v^\ep$ be defined by \eqref{eq:barrier defn}, with $0<\tilde\sigma<1$, $R>1$,  $\delta=\delta(\ep)$ satisfying \eqref{delta:def}, and  
\begin{equation}\label{d_i(0,x)condition_lem}
 d_i(0,x)=d_i^0(x)+o_\ep(1),
\end{equation}
  where $d_i^0$ is given by \eqref{eq:initial d_i}. 
Then, there exists $\sigma_0>0$ such that, for all $0<\tilde\sigma<\sigma_0$, there is $\ep_0=\ep_0(\tilde\sigma)>0$ such that for all  $0<\ep<\ep_0$,
\begin{equation}\label{eq:initial_inequality}
v^\ep(0,x) \leq u^\ep_0(x) \qquad \text{for all } x \in \R^n,
\end{equation}
where $u_0^\ep$ is the initial datum defined in \eqref{initial_data}.
\end{lem}

\begin{proof}

 By \eqref{Omega_0^iassumptions}, there exists $\ell_0>0$ such that, for $i=1,\ldots,N-1$,  
 \begin{equation}\label{d_oell0cond} 
 d_{i}^0(x)\geq  d_{i+1}^0(x)+2\ell_0 \qquad \text{for all } x \in \R^n
\end{equation}
It follows from \eqref{d_i(0,x)condition_lem} that, for $\ep$ sufficiently small
and $i=1,\ldots,N-1$,
\begin{equation}\label{d_oell0condbis} 
 d_{i}(0,x)\geq  d_{i+1}(0,x)+\ell_0 \qquad \text{for all } x \in \R^n.
\end{equation}
Fix $x\in\R^n$. We consider two cases.
\medskip

\noindent
{\bf Case 1}:~\emph{There exists $i_0\in\{1,\ldots,N\}$ such that $|d_{i_0}^0(0,x)-\tilde\sigma^N| < 2\delta$.}

Since $\delta=o_\ep(1)$, we may choose $\tilde\sigma$ and $\ep$ so small that $\delta<\tilde\sigma<\ell_0/4$. Then,  by  \eqref{d_oell0condbis}, we have
$$d_i(0,x)-\tilde\sigma^N\ge \tilde\sigma\quad\text{for }i=1,\ldots,i_0-1,$$
$$d_i(0,x)-\tilde\sigma^N\le -\tilde\sigma\quad\text{for }i=i_0+1,\ldots,N.$$
Therefore, by \eqref{eq:asymptotics for phi}, 
$$\phi\!\left(\frac{d_i(0,x)-\tilde\sigma^N}{\ep}\right) \leq \begin{cases}
\displaystyle 1+C\frac{\ep^{2s}}{\tilde\sigma^{2s}} & \text{for }i=1,\ldots,i_0-1,\\[4pt]
\displaystyle C\frac{\ep^{2s}}{\tilde\sigma^{2s}} & \text{for }i=i_0+1,\ldots,N.
\end{cases}.$$
By \eqref{psi_far}, and using that $\delta<\tilde\sigma$ and $\ep/\delta^2=o_\ep(1)$,  we obtain, for $i\neq i_0$, 
$$\ep^{2s}\psi_i\!\left(\frac{d_i(0,x)-\tilde\sigma^N}{\ep};0,x\right)\leq C\frac{\ep^{4s}}{\delta^{2s}\tilde\sigma^{2s}}\leq C\ep^{2s}.$$
For the index $i_0$, it follows from the monotonicity of $\phi$  and \eqref{eq:asymptotics for phi} that
$$\phi\!\left(\frac{d_{i_0}(0,x)-\tilde\sigma^N}{\ep}\right)\le \phi\!\left(\frac{2\delta}{\ep}\right)\le 1-C\,\frac{\ep^{2s}}{\delta^{2s}}.$$
If $|d_{i_0}(0,x)-\tilde\sigma^N|<\delta$, then, by \eqref{psi_near},
$$\ep^{2s}\!\left|\psi_{i_0}\!\left(\frac{d_{i_0}(0,x)-\tilde\sigma^N}{\ep};0,x\right)\right|\le C\ep^{2s}. $$
On the other hand,  if $\delta\le|d_{i_0}(0,x)-\tilde\sigma^N|<2\delta$, then by \eqref{psi_far} and the fact that $\ep/\delta^2=o_\ep(1)$, 
$$\ep^{2s}\!\left|\psi_{i_0}\!\left(\frac{d_{i_0}(0,x)-\tilde\sigma^N}{\ep};0,x\right)\right|\le C\,\frac{\ep^{4s}}{\delta^{4s}}\leq C\ep^{2s}.$$
Combining the above estimates, the bound on $\bar a_\ep^i$  provided by Lemmas~\ref{lem:a_estimates}, and recalling that
$\mu\ge 0$, we obtain
\begin{equation}\label{vep(0,x)esticase1}v^\ep(0,x) \le i_0-C\frac{\ep^{2s}}{\delta^{2s}}+C\ep^{2s} \le i_0-C\frac{\ep^{2s}}{2\delta^{2s}}.
\end{equation}
We now estimate $u_0^\ep$. As above, using
\eqref{d_oell0cond}, \eqref{eq:asymptotics for phi}, and the fact that $\phi\ge0$, we obtain
$$\phi\!\left(\frac{d_i^0(x)}{\ep}\right) \;\ge\;
\begin{cases}
\displaystyle 1 - C\frac{\ep^{2s}}{\tilde\sigma^{2s}} & \text{for }i=1,\ldots,i_0-1,\\[4pt]
0 & \text{for }i=i_0+1,\ldots,N.
\end{cases}$$
 Next, by \eqref{d_i(0,x)condition_lem}
  and  $\delta=o_\ep(1)$, for $\ep$ sufficiently small, we have 
  $$d_{i_0}^0(x)\ge\tilde\sigma^N-2\delta+o_\ep(1)\ge\frac{{\tilde\sigma}^N}{2}.$$
  Therefore, by \eqref{eq:asymptotics for phi}, 
$$\phi\!\left(\frac{d_{i_0}^0(x)}{\ep}\right) \ge 1-C\frac{\ep^{2s}}{\tilde\sigma^{2sN}}.$$
We conclude that
$$u_0^\ep(x)\geq i_0-C\frac{\ep^{2s}}{\tilde\sigma^{2sN}}.$$
Combining this with \eqref{vep(0,x)esticase1} and using that $\delta=o_\ep(1)$, we obtain
$$v^\ep(0,x)  \le i_0-C\frac{\ep^{2s}}{2\delta^{2s}}\le i_0-C\frac{\ep^{2s}}{\tilde\sigma^{2sN}}\leq u_0^\ep(x),$$
provided $\ep$ is sufficiently small.

This proves \eqref{eq:initial_inequality} in Case 1. 

\medskip

\noindent{\bf Case 2}:~\emph{$|d_i(0,x)-\tilde\sigma^N| \geq 2\delta$ for all $i=1,\ldots,N$.}

By \eqref{d_i(0,x)condition_lem}, for $\ep$ sufficiently small, 
$$d_i(0,x)-\tilde\sigma^N\leq d_i^0(x).$$
Hence, by the monotonicity of $\phi$, 
$$\phi\!\left(\frac{d_i(0,x)-\tilde\sigma^N}{\ep}\right)\le \phi\!\left(\frac{d_i^0(x)}{\ep}\right).$$
Arguing as in Case~1, it follows that
$$\ep^{2s}\!\left|\psi_i\!\left(\frac{d_i(0,x)-\tilde\sigma^N}{\ep};0,x\right)\right|\le C\ep^{4s}.$$ 
Recalling the definition of $\mu$  in  \eqref{def:mu_definition} and \eqref{alphasigma}, we have 
 $$\frac{\ep^{2s}}{\alpha}\mu[d_1-\tilde\sigma^N,\ldots,d_N-\tilde\sigma^N](0,x)=\frac{\ep^{2s}\tilde\sigma}{\delta^{2s}}.$$
 Combining the above estimates with the bound on $\bar a_\ep^i$  provided by Lemmas~\ref{lem:a_estimates}, and using that $\delta=o_\ep(1)$, we obtain, for $\ep$ sufficiently  small,
\begin{align*}
v^\ep(0,x) \le u^\ep_0(x)+C\ep^{2s}-\frac{\ep^{2s}\tilde\sigma}{\delta^{2s}} \le u^\ep_0(x).
\end{align*}
This proves \eqref{eq:initial_inequality} in Case 2.
\end{proof}


\section{Proof of Theorem \ref{thm:main_result}}\label{sec:proof main_result}

In this section, we complete the proof of Theorem  \ref{thm:main_result} using the barriers constructed in Section \ref{barriersection},  the comparison principle and the decay estimates and bounds established in Section \ref{sec:corrector-section}.
\begin{proof}[Proof of Theorem 1.1]
Set
\[
 U(t,x):=\sum_{i=1}^N\mathbf 1_{\Omega_t^i}(x),
 \qquad
 \mathcal G:=\bigcup_{i=1}^N
 \{(t,x)\in[0,T]\times\mathbb R^n:x\in\Gamma_t^i\}.
\]
We shall prove that $u^\varepsilon\to U$ locally uniformly on
$([0,T]\times\mathbb R^n)\setminus\mathcal G$.

First we construct strict geometric barriers to which Section \ref{barriersection} applies.
Let $d_i(t,x)$ be the signed distance function to $\Omega_t^i$. Then, 
by \eqref{eq:velocity-intro}  and \eqref{eq:fmc_interaction},
\begin{equation}\label{mainthm:eqfronts}
\partial_td_i(t,x)=c_0\left(\kappa[x,d_i(t,\cdot)]
              +\sum_{j\ne i}P[d_j](t,x)\right),\qquad x\in \Gamma_t^i,
\end{equation}
where $P$ is defined  in \eqref{def:P}. 
Smoothness and strict separation on the compact time interval give a
uniform tubular radius and a positive lower bound for the distances
between distinct fronts. 

For a small $r\ge0$, consider the parallel sets
\[
 \Omega_t^{i,-}(r):=\{d_i(t,\cdot)>r\},
 \qquad
 \Omega_t^{i,+}(r):=\{d_i(t,\cdot)>-r\},
\]
and let $\tilde d_i^{\pm,r}$ be their signed distances. 
Then 
 \[\tilde{d}_i^{-,r}=d_i-r,
 \qquad \tilde{d}_i^{+,r}=d_i+r.
\]
By the smoothness of the flow $(\Omega_t^i)_{i=1}^N$, Proposition \ref{lem:kappa-neighborhood} and \eqref{mainthm:eqfronts}, for $\sigma>0$ and $r$ sufficiently small, and $\tilde\sigma$ defined as in \eqref{alphasigma}, there exists $C>0$ such that 
\begin{equation}\label{proof:geometric-error}
\left|\partial_td_i(t,x)-c_0\left(\kappa[x,d_i(t,\cdot)]
              +\sum_{j\ne i}P[d_j](t,x)\right)\right|\leq C(r+\sigma) ,\qquad |\tilde d_i^{\pm,r}(t,x)|< 2\tilde\sigma.
\end{equation}
Given $\sigma,\,A>0$, we choose 
\[
 r=r(t):=\sigma(e^{At}-1).
\]
For $\sigma$ sufficiently small, all the parallel families with
$r=r(t)$ remain smooth and strictly nested on $[0,T]$.
Let $d_i^{\pm,r}$ be the smooth extensions of $\tilde{d}_i^{\pm,r}$ supplied
by Definition \ref{defn:extension} with $\tau$ chosen as in  Lemmas~\ref{lem:a_estimates} and \ref{lem:Pep_estimates}. By \eqref{approxdtau},
\begin{equation}\label{proof:distance-approximation}
d_i^{\pm,r}=\tilde{d}_i^{\pm,r}+o_\ep(1).
\end{equation}
Therefore, for  $| d_i^{\pm,r}(t,x)|< 2\tilde\sigma$, recalling Remark \ref{Kdextensionrem}, we have 
$$\partial_t d_i^{\pm,r}(t,x)=\partial_t \tilde d_i^{\pm,r}(t,x)
=\partial_t d_i(t,x)\pm r'(t)=\partial_t d_i(t,x)\pm A(r(t)+\sigma),$$
$$\kappa[x,d_i^{\pm,r}(t,\cdot)]=\kappa[x,\tilde d_i^{\pm,r}(t,\cdot)]=\kappa[x,d_i(t,\cdot)].$$
Arguing as in the proof of \eqref{Plipschitz}, we also have
$$\sum_{j\ne i}P[d_j](t,x)=\sum_{j\ne i}P[d^{\pm,r}_j](t,x)+O(r(t))+o_\ep(1).$$
Plugging into \eqref{proof:geometric-error} and choosing $A$ sufficiently large (but independent of all other parameters), we obtain  
\begin{equation*}
\partial_td_i^{-,r}(t,x)\leq c_0\left(\kappa[x,d_i^{-,r}(t,\cdot)]
              +\sum_{j\ne i}P[d^{-,r}_j](t,x)-\sigma\right),\qquad | d_i^{-,r}(t,x)|< 2\tilde\sigma,
\end{equation*}
\begin{equation*}
\partial_td_i^{+,r}(t,x)\geq c_0\left(\kappa[x,d_i^{+,r}(t,\cdot)]
              +\sum_{j\ne i}P[d^{+,r}_j](t,x)+\sigma\right),\qquad | d_i^{+,r}(t,x)|< 2\tilde\sigma.
\end{equation*}
Moreover,
since $r(0)=0$, 
\[
 d_i^{\pm,r}(0,x)
 =d_i^0(x)+o_\varepsilon(1).
\]
Define
\begin{equation*}
\begin{split}
v_-^{\ep,r}(t,x) =&\sum_{i=1}^N \phi\left( \frac{d_i^{-,r}(t,x)- \tilde{\sigma}^N }{\ep}\right)+E_-^{\ep,r}(t,x)
\end{split}
\end{equation*}
where $E_-^{\ep,r}$ denote all the corrector terms in \eqref{eq:barrier defn} associated to the distance functions $d_1^{-,r},\ldots, d_N^{-,r}$. 
For a suitable choice of the parameters, by Lemmas \ref{lem:barrier} and \ref{lem:initial_sub}, 
\begin{equation*}
\ep \partial_t v_-^{\ep,r} - \mathcal{I}_n^s[v_-^{\ep,r} ] +\frac{1}{\ep^{2s}} W'(v_-^{\ep,r} ) \leq 0 \quad \hbox{in}~(0,T) \times \R^n,
\end{equation*}
and 
$$v_-^{\ep,r}(0,x)\leq u_0(x). $$ 
The analogous upper barrier, using $\mu$ instead of $-\mu$ in the
corrector construction and the positive spatial shift, has the form
\[
 v_+^{\varepsilon,r}(t,x)
 =\sum_{i=1}^N\phi\left(
   \frac{d_i^{+,r}(t,x)
                +\widetilde\sigma^N}{\varepsilon}\right)
   +E_+^{\varepsilon,r}(t,x).
\]
Consequently, the comparison principle implies
\begin{equation}\label{proof:comparison}
 v_-^{\varepsilon,r}\le u^\varepsilon
 \le v_+^{\varepsilon,r}
 \qquad\text{on }[0,T]\times\mathbb R^n.
\end{equation}
Now, let $ K$ be any compact subset of
$([0,T]\times\mathbb R^n)\setminus\mathcal G$. Then
\[
 m:=\min_{1\le i\le N}\inf_{(t,x)\in K}|d_i(t,x)|>0.
\]
The  signed distances to the parallel sets satisfy
\[
 \|\tilde d_i^{\pm,r}-d_i\|_\infty
 \le r(T).
\]
Choose  $\sigma$ small enough that
$r(T)+\widetilde\sigma^N<m/4$.
By \eqref{proof:distance-approximation}, for sufficiently small
$\varepsilon$, the quantities $
 d_i^{-,r}-\widetilde\sigma^N$, 
  $d_i^{+,r}+\widetilde\sigma^N$ 
have the same sign as $d_i$ on $ K$ and absolute value at
least $m/2$. The layer estimate in Lemma 5.1 therefore implies
\[
 \sup_{(t,x)\in K}
 \left|\phi\left(
 \frac{d_i^{\pm,r}(t,x)
                  \pm\widetilde\sigma^N}{\varepsilon}\right)
       -\mathbf 1_{\Omega_t^i}(x)\right|
 \le C\frac{\varepsilon^{2s}}{m^{2s}}\longrightarrow0\qquad \text{as }\ep\to0.
\]
Moreover, by  Lemmas~\ref{lem:a_estimates}, \ref{lem:psi-reg} and \eqref{def:mu_definition}
\[
 \|E_-^{\varepsilon,r}\|_\infty
 +\|E_+^{\varepsilon,r}\|_\infty
 \le C\varepsilon^{2s}
                \(1+\frac{\ep^{2s}}{\delta^{2s}m^{2s}}+\frac{1}{\delta^{2s}}\)\longrightarrow0\qquad \text{as }\ep\to0.
\]
We used $\varepsilon/\delta\to0$, which follows from \eqref{delta:def}.
 This proves
$
 \|v_\pm^{\varepsilon,r}-U\|_{L^\infty( K)}
 \longrightarrow0,
$
which combined with \eqref{proof:comparison} yields
$\|u^\varepsilon-U\|_{L^\infty( K)}\to0$.

Finally, strict nesting gives $U=N$ on $\Omega_t^N$,
$U=i$ on $\Omega_t^i\setminus\overline{\Omega_t^{i+1}}$
for $1\le i<N$, and $U=0$ on
$\mathbb R^n\setminus\overline{\Omega_t^1}$.
This establishes the convergence, with the interfaces excluded,
and completes the proof.

\end{proof}

\section{Proof of Lemma \ref{lem:interaction asymptotics}}\label{sec:proof of interaction asymptotics}
For ease of notation, throughout this section we omit the dependence on $t$.

Assume that $d(x) > \ell$. By definition,
\[
\mathcal{I}_n^s\left[\phi\left(\frac{d}{\ep}\right)\right](x) = \int_{\R^n} \left(\phi\left(\frac{d(x+z)}{\ep}\right) - \phi\left(\frac{d(x)}{\ep}\right)\right) \frac{dz}{|z|^{n+2s}}.
\]
Let $c_1:=\|\nabla d\|_\infty$. We split the integral as
\[
\int_{\R^n} \left(\phi\left(\frac{d(x+z)}{\ep}\right) - \phi\left(\frac{d(x)}{\ep}\right)\right) \frac{dz}{|z|^{n+2s}} = \int_{\{|z| < \frac{\ell}{2c_1}\}} (\cdots) + \int_{\{|z| \geq \frac{\ell}{2c_1}\}} (\cdots) =: I + II.
\]
We first estimate $I$. By the mean value theorem, for some $\tau \in (0,1)$,
\[
\phi\left(\frac{d(x+z)}{\ep}\right) - \phi\left(\frac{d(x)}{\ep}\right) = \dot{\phi}\left(\frac{d(x)}{\ep} + \tau \frac{d(x+z) - d(x)}{\ep}\right) \cdot \frac{d(x+z) - d(x)}{\ep}.
\]
If $d(x) > \ell$ and $|z| < \frac{\ell}{2c_1}$, then 
\[
\frac{d(x)}{\ep} + \tau \frac{d(x+z) - d(x)}{\ep}\geq\frac{\ell}{2\ep},
\]
and by \eqref{eq:asymptotics for phi dot}, 
\[
\left|\phi\left(\frac{d(x+z)}{\ep}\right) - \phi\left(\frac{d(x)}{\ep}\right)\right| \leq C \frac{\ep^{2s}}{\ell^{2s+1}} |z|.
\]
Hence, if $\ep / \ell^2 = o_\ep(1)$, 
\begin{equation}\label{eq:I_est_interaction}
|I| \leq C \frac{\ep^{2s}}{\ell^{2s+1}} \int_{\{|z| < \frac{\ell}{2c_1}\}} |z|^{1-n-2s} \, dz = C\frac{\ep^{2s}}{\ell^{4s}} = o_\ep(1). 
\end{equation}
Next, we estimate $II$. We  decompose the integral according to the sign of $d(x+z)$:
\begin{align*}
II &= \int_{\{|z| \geq \frac{\ell}{2c_1},\, d(x+z) > \ep^{\frac12}\}} (\cdots) + \int_{\{|z| \geq \frac{\ell}{2c_1},\, |d(x+z)| \leq \ep^{\frac12}\}} (\cdots) + \int_{\{|z| \geq \frac{\ell}{2c_1},\, d(x+z) < -\ep^{\frac12}\}} (\cdots) \\
&=: II_1 + II_2 + II_3.
\end{align*}
If $d(x+z) > \ep^{\frac12}$ and   $d(x) > \ell$, then by \eqref{eq:asymptotics for phi},
$$\phi\left(\frac{d(x+z)}{\ep}\right) - \phi\left(\frac{d(x)}{\ep}\right)=(1+O(\ep^s))-\(1+O\(\frac{\ep^{2s}}{\ell^{2s}}\)\)=O(\ep^s)+O\(\frac{\ep^{2s}}{\ell^{2s}}\).$$
Consequently,  
\[
|II_1|\leq \(O(\ep^s)+O\(\frac{\ep^{2s}}{\ell^{2s}}\)\)\int_{\{|z| \geq \frac{\ell}{2c_1}\}}\frac{dz}{|z|^{n+2s}}=O\(\frac{\ep^s}{\ell^{2s}}\)+O\(\frac{\ep^{2s}}{\ell^{4s}}\)=o_\ep(1),
\]
again if $\ep/\ell^2=o_\ep(1)$.

We now estimate $II_2$. Since the zero level set of $d$ is a smooth surface, we have $|\{z\,:\, |d(z+x)|<\ep^\frac12\}|\leq C\ep^\frac{1}{2}$. Thus,
$$|II_2|\leq 2\int_{\{|z| \geq \frac{\ell}{2c_1},\, |d(x+z)| \leq \ep^{\frac12}\}}\frac{dz}{|z|^{n+2s}}\leq \frac{C}{\ell^{n+2s}}\int_{\{|d(x+z)| \leq \ep^{\frac12}\}}dz\leq \frac{C\ep^\frac{1}{2}}{\ell^{n+2s}}=o_\ep(1), $$
provided  $\ell$ satisfies $\ep^{\frac12}/\ell^{n+2s}=o_\ep(1)$.

We finally estimate $II_3$. If $d(x+z)<-\ep^\frac{1}{2}$ and $d(x)>\ell$, then $|z|\ge\frac{\ell}{2c_1}$. Moreover, by \eqref{eq:asymptotics for phi},
$$\phi\left(\frac{d(x+z)}{\ep}\right) - \phi\left(\frac{d(x)}{\ep}\right)=-1+O(\ep^s)+O\(\frac{\ep^{2s}}{\ell^{2s}}\).$$
Therefore, arguing as above,
\[
II_3= \int_{\{d(x+z)<-\ep^{\frac12}\}}(\cdots)=-\int_{\{ d(x+z)<-\ep^{\frac12}\}}\frac{dz}{|z|^{n+2s}}+o_\ep(1)=-\int_{\{ d(x+z)<0\}}\frac{dz}{|z|^{n+2s}}+o_\ep(1).
\]

Combining the above estimates, if $d(x) > \ell$, the only non-trivial contribution comes from $II_3$, and therefore
\[
\mathcal{I}_n^s\left[\phi\left(\frac{d}{\ep}\right)\right](x) = -\int_{\{d(x+z) < 0\}} \frac{dz}{|z|^{n+2s}} + o_\ep(1).
\]
Similarly, if $d(x) < -\ell$, the only non-trivial contribution comes from $II_1$, yielding
\[
\mathcal{I}_n^s\left[\phi\left(\frac{d}{\ep}\right)\right](x) = +\int_{\{d(x+z) > 0\}} \frac{dz}{|z|^{n+2s}} + o_\ep(1).
\]
 This completes the  proof of  \eqref{eq:interaction limit}.

\section{Proof of Lemma \ref{lem:Pep_estimates}}\label{sec:proof of Pep estimates}
 Estimates \eqref{Pep_bound}  and \eqref{eq:interaction limitbis} follow immediately from \eqref{eq:interaction limit} and the definition of $g$ in \eqref{gdefforPep}. 

We next estimate the derivatives of $P_\ep[d]$. We will establish the estimate for $\nabla_x P_\ep$; the estimate for $\partial_t P_\ep$ follows by a similar argument. Differentiating and using \eqref{Pep_bound}, we obtain
\begin{align}\label{partialPcomput1}
    \nabla_x P_\ep[d] = \nabla_x \mathcal{I}_n^s\left[\phi\(\frac{d}{\ep}\)\right] \cdot g + O(1).
\end{align}
We then compute 
\begin{equation}\label{partialxba-close}\begin{split}
   \partial_{x_i}\mathcal{I}_n^s\left[\phi\(\frac{d}{\ep}\)\right] = & \int_{\R^n} \bigg[ \dot{\phi} \left( \frac{d(x+ z)}{\ep}\right) \frac{\partial_{x_i}d(x+z)}{\ep} - \dot{\phi}\left( \frac{d(x)}{\ep}  \right) \frac{\partial_{x_i}d(x)} {\ep} \bigg] \frac{dz}{|z|^{n+2s}}  \\
   =  &\ep^{-1}\left\{ \int_{\R^n} \bigg [ \dot{\phi} \left( \frac{d(x+z)}{\ep}\right) (\partial_{x_i}d(x+ z) - \partial_{x_i}d(x))\right.  \\
    & + \left(  \dot{\phi} \left( \frac{d(x+ z)}{\ep}\right) -  \dot{\phi}\left( \frac{d(x) }{\ep}\right) \right) \partial_{x_i}d(x).
\end{split}
\end{equation}
Since $g=0$ whenever $|d(x)|\leq \ell_0/2,$ it suffices to estimate \eqref{partialxba-close} when $|d(x)|\geq \ell_0/2$.

Recall the uniform Lipschitz estimates and the estimates on $D^2d$, depending on the parameter $\tau$, given in \eqref{approxdtau}. 
Since $d$ is Lipschitz continuous, there exists $c>0$ such that if $|z|\leq c\ell_0$ then  $|d(x+z)-d(x)|\leq  \ell_0/4$, and consequently  $|d(x+z)|\geq \ell_0/4$.
We therefore split 
\begin{align}\label{partialxbnablaafar}
  \partial_{x_i}\mathcal{I}_n^s\left[\phi\(\frac{d}{\ep}\)\right]  &=\ep^{-1}\left(\int_{\{|z|\leq c\ell_0\}}(\ldots)+\int_{\{|z|>c\ell_0\}}(\ldots)\right)=:\ep^{-1}(I+II).
   \end{align}
  We first estimate $I$. Using the asymptotic estimates  for $\dot{\phi}$ and $\ddot{\phi}$ given in \eqref{eq:asymptotics for phi dot},  for $|z|\leq c\ell_0$ we have 
$$0\leq  \dot{\phi} \left(\frac{d(x+z)}{\ep}\right)\leq C\frac{\ep^{2s+1}}{\ell_0^{2s+1}},$$
and, for some $\theta\in (0,1)$,  
 $$\left|  \dot{\phi} \left( \frac{d(x+ z)}{\ep}\right) -  \dot{\phi}\left( \frac{d(x) }{\ep}\right) \right| \leq 
  \left|\ddot{\phi}\left(  \frac{d(x)}{\ep}  +\theta \frac{d(x+z)-d(x)}{\ep} \right)\right|\frac{|z|}{\ep}\leq C\frac{\ep^{2s}}{\ell_0^{2s+1}}.$$  
  Therefore, using \eqref{approxdtau}, 
\begin{equation}\label{Inabla_xafar}\begin{split}
    |I| \le 
    & C \frac{\ep^{2s+1}}{\ell_0^{2s+1}\tau} \int_{ \{ |z| \le c \ell_0 \} } \frac{dz}{|z|^{n+2s-1}} + C \frac{\ep^{2s}}{\ell_0^{2s+1}} \int_{ \{|z| \le  c \ell_0\}} \frac{dz}{|z|^{n+2s-2}} \\
    \leq& C\left( \frac{\ep^{2s+1}}{\ell_0^{4s}\tau}+\frac{\ep^{2s}}{\ell_0^{4s-1}}\right). 
\end{split}
\end{equation}  
We next estimate $II$. We have
\begin{equation}\label{IIsplitnablaxfar}\begin{split}
    |II|& \le  C \int_{\{ |z|>c\ell_0  \}} \dot{\phi} \left( \frac{d(x+ z)}{\ep}\right) \frac{dz}{|z|^{n+2s}} + C\int_{\{|z|>c\ell_0 \}} \dot{\phi} \left( \frac{d(x)}{\ep} \right)  \frac{dz}{|z|^{n+2s}} \\&=:II_1 + II_2. 
\end{split}
\end{equation}
We further split
\begin{align*}
    II_1 & =\int_{ \{  |z|>c\ell_0,\, |d(x+z)| \le\ep^{1\!/2}\} }(\dots) + \int_{ \{  |z|>c\ell_0,\, |d(x+z)| > \ep^{1\!/2}\} }(\dots) \\
    & =: J_1+J_2.
\end{align*} 
Using that $\{z\,:\,d(x+z)=0\}$ is a smooth surface, we get
\begin{align*}
     J_1 \le   C \int_{ \{  |z|>c\ell_0,\, |d(x+z)| \le \ep^{1\!/2}\} } \frac{dz}{|z|^{n+2s}} \le \frac{C}{\ell_0^{n+2s}} \int_{ \{|d(x+z)| \le \ep^{1\!/2}\} } dz \le 
    \frac{C \ep^{\frac12}}{\ell_0^{n+2s}}.
\end{align*}
By estimate \eqref{eq:asymptotics for phi dot} for $\dot{\phi}$, we also have
\begin{align*}
     J_2 \le C\ep^{s+\frac12} \int_{ \{  |z|>c\ell_0\} } \frac{dz}{|z|^{n+2s}} = C \frac{\ep^{s+\frac12}}{\ell_0^{2s}}. 
\end{align*}
 From the above estimates on $J_1$ and $J_2$ we infer that 
   \begin{equation*}\label{II_1estimnablaax_far}II_1\leq \frac{C \ep^{\frac12}}{\ell_0^{n+2s}}.
   \end{equation*}
We finally estimate $II_2$. By  \eqref{eq:asymptotics for phi dot}  for $\dot{\phi}$ we have 
\begin{equation*}\label{II_2est_nablax_far_zerpgra}II_2= C\dot{\phi} \left( \frac{d(x)}{\ep} \right)  \int_{\{  |z|>c\ell_0\}}\frac{dz}{|z|^{n+2s}}\leq \frac{C\ep^{2s+1}}{\ell^{2s+1} }\int_{\{  |z|>c\ell_0\}}\frac{dz}{|z|^{n+2s}}\leq\frac{C\ep^{2s+1}}{\ell_0^{4s+1} }.
\end{equation*}
Combining the estimates for  $II_1$ and $II_2$, we obtain 
$$|II|\leq C \frac{ \ep^{\frac12}}{\ell_0^{n+2s}}.$$
Together with \eqref{partialPcomput1}, \eqref{partialxbnablaafar}, and \eqref{Inabla_xafar},  this estimate implies that there exist $\tau_\ep=o_\ep(1)$ such that if $\tau\geq \tau_\ep$ in  Definition~\ref{defn:extension}, then \eqref{Pep_deriv_bound} holds. The proof of the lemma is complete.

\appendix

\section{Examples of the interacting flow}\label{sec:examples}

In this appendix, we examine two simple configurations that illustrate the structure and dynamics of the coupled system \eqref{eq:fmc_interaction}.

\subsection*{Example 1: Half-spaces}

Let $\Omega_t^i$, $i=1,\ldots,N$, be the half-spaces orthogonal to the $e_1$-axis defined by
$$
\Omega_t^i=\{x\cdot e_1>x^i(t)\},\qquad t \geq 0,
$$
where
\begin{equation}\label{eq1-halfspaces}
x^1(t)<x^2(t)<\cdots<x^N(t),\qquad t\geq 0.
\end{equation}
Here, $x^i(t)$ denotes the position along the $e_1$-axis of the interface
$\Gamma_t^i=\{x\cdot e_1=x^i(t)\}.$

By \eqref{eq1-halfspaces}, the family $(\Omega_t^i)_{i=1}^N$ satisfies the nesting condition $\overline{\Omega^{i+1}_t}\subset \Omega^i_t$.

The signed distance function to $\Omega_t^i$ is $d_i(t,x)=x\cdot e_1-x^i(t).$
Notice that
$\kappa[x,d_i]=0$ and $\partial_t d_i=-\dot{x}^i$,  $i=1,\ldots,N.$ 

Fix $x\in\Gamma_t^i$, so that $x\cdot e_1=x^i(t)$. If $j>i$, then
$  d_j(t,x)=x^i(t)-x^j(t)<0,$
and the contribution of $\Omega_t^j$ to the interaction term in \eqref{eq:fmc_interaction} is
\begin{align*}
\int_{\{d_j(t,x+z)>0\}}\frac{dz}{|z|^{n+2s}}
&=\int_{\{z_1>x^j(t)-x^i(t)\}}\frac{dz}{|z|^{n+2s}}=C_{n,s}\int_{x^j(t)-x^i(t)}^\infty
\frac{dz_1}{z_1^{1+2s}}\\
&=\frac{C_{n,s}}{2s}
\frac{1}{\bigl(x^j(t)-x^i(t)\bigr)^{2s}},
\end{align*}
where $C_{n,s}$ is defined in \eqref{eq:Cns}. The case $j<i$ is analogous.

Consequently, system \eqref{eq:fmc_interaction} reduces to
\begin{equation}\label{eq3-halfspaces}
\dot{x}^i(t)
=\frac{c_0C_{n,s}}{2s}
\sum_{j\neq i}
\frac{x^i(t)-x^j(t)}{|x^i(t)-x^j(t)|^{2s+1}},
\qquad i=1,\ldots,N.
\end{equation}
This is precisely the system of ODEs governing one-dimensional dislocation dynamics. 
Given initial condition $x^i(0)=x_0^i$ with $x^1_0<\cdots <x_0^N$, 
 the existence and uniqueness of a global in time solution $(x^1(t),\ldots,x^N(t))$ to \eqref{eq3-halfspaces} satisfying the ordering condition \eqref{eq1-halfspaces} were established in \cite{GonzalezMonneau, VanMeursPeletierPozar}.

\subsection*{Example 2: Concentric spheres} 
Let $\Omega_t^i$, $i=1,\ldots,N$, be the concentric balls  defined by
$$
\Omega_t^i=B_{R_i(t)}(0),\qquad t> 0,
$$
where
\begin{equation}\label{eq1-balls}
R_N(t)<R_{N-1}(t)<\cdots<R_1(t),\qquad t\geq 0.
\end{equation}
Here the interfaces are the spheres $\Gamma_t^i = \partial B(0, R_i(t))$. 
Condition \eqref{eq1-balls} guarantees that \eqref{nestingassumption} is satisfied. 
The signed distance functions are $d_i(t,x) = R_i(t) - |x|$. 

Fix $x \in \Gamma_t^i$, so that $|x| = R_i$. By Proposition~\ref{ballFMC}, $\kappa[x,d_i] =  -\omega/R_i^{2s}$ for a constant $\omega> 0$. Moreover, $\partial_t d=\dot R_i$. 
 By rotational invariance, the contribution of $\Omega_t^j$, $j\neq i$,  to the interaction term in \eqref{eq:fmc_interaction} depends only on the radii $R_i(t)$ and $R_j(t)$, and  splitting according to the sign of $d_j(t,x) = R_j(t) - R_i(t)$,  it can be written 
\begin{equation}\label{eq:sphere_interactions}
P(R_i(t),R_j(t)):= \int_{\{d_j(t,z)\operatorname{sgn}(d_j(t,x))<0\}}\frac{dz}{|z-x|^{n+2s}}
=\begin{cases}
\displaystyle\int_{\{|z|>R_j(t)\}}\frac{dz}{|z-x|^{n+2s}} & \text{if } j<i,\\[15pt]
\displaystyle\int_{\{|z|<R_j(t)\}}\frac{dz}{|z-x|^{n+2s}} & \text{if } j>i.
\end{cases}
\end{equation}
Consequently,  system \eqref{eq:fmc_interaction} reduces to the following ODE system 
\begin{equation}\label{eq:sphere_ODE}
\dot R_i(t) = c_0\left(-\frac{\omega}{R_i(t)^{2s}} + \sum_{j\neq i} \operatorname{sgn}\!\big(R_i(t) - R_j(t)\big)\, P(R_i(t),R_j(t))\right),\qquad i = 1, \dots, N.
\end{equation}
The first term on the right-hand side of \eqref{eq:sphere_ODE} is well defined as long as $R_i(t)>0$. The interaction terms are well defined as long as the radii remain pairwise distinct, in particular, this is guaranteed by the ordering condition \eqref{eq1-balls}.
If two radii coincide, the singularity of the kernel $|z-x|^{-(n+2s)}$ in \eqref{eq:sphere_interactions} is no longer integrable.

\subsection*{No collision of spheres.} We now show that along the flow \eqref{eq:sphere_ODE}, two adjacent spheres never touch.  
Given an initial configuration 
\begin{equation}\label{initialcond:spheres}
R_i(0)=R_i^0,\quad\text{with}\qquad R^0_1 > \cdots > R^0_N>0,
\end{equation}
we will show that the ordering \eqref{eq1-balls} is preserved and the configuration can degenerate only when the innermost sphere goes extinct.

Denote by
\begin{align*}
    \mathcal{D} := \left\{(R_1, \ldots, R_N) \in \mathbb{R}^N: R_1 > R_2> \cdots >R_N >0 \right\}.
\end{align*}
\begin{lem}\label{lem:non_collision}
    Let $(R_1,\ldots,R_N) \in C^1([0,T);\mathcal{D})$ be any solution of \eqref{eq:sphere_ODE}. 
    For $t \in [0,T)$, define 
    \begin{align*}
        \gamma(t) :=\min_{1 \le k \le N-1} (R_k(t) - R_{k+1}(t)).
    \end{align*}
    Then, $\gamma$ is a monotone nondecreasing function in $[0,T)$. 
\end{lem}

\begin{proof}
Since each $R_k$ is continuously differentiable on $[0,T)$, the function $\gamma$, being the minimum of finitely many continuously differentiable functions, is locally Lipschitz continuous and therefore differentiable almost everywhere. Let $t\in(0,T)$ be a point at which $\gamma$ is differentiable, and choose $i\in\{1,\ldots,N-1\}$ such that
$$\gamma(t)= R_i(t) - R_{i+1}(t).$$
We will show that 
\begin{equation}\label{eq:non_collision}
        \dot{R}_i(t) - \dot{R}_{i+1}(t) \geq c_0 \omega \left( \frac{1}{R_{i+1}(t)^{2s}} - \frac{1}{R_i(t)^{2s}} \right)>0.
    \end{equation}
   For simplicity, we drop the dependence on $t$ in what follows.
     Let $e \in \mathbb{S}^{n-1}$ and consider the two points on the closest pair of spheres on a common ray through $e$,
    \begin{align*}
        x:= R_i e, \qquad x':= R_{i+1}e, \qquad |x-x'| = R_i - R_{i+1} = \gamma.
    \end{align*}

\noindent
For $R>0$, denote the interaction kernel
\begin{equation}\label{eq:spheres_kernel}
    K_R(z):= \frac{1}{|z-Re|^{n+2s}} \geq 0.
\end{equation}
Recall definition \eqref{eq:sphere_interactions}, and set $P_{i,j}:=P(R_i,R_j)$.  By rotational invariance, we can write
\begin{align*}
    P_{i,i+1} = \int_{ \{|z| < R_{i+1} \} } K_{R_i}(z)\, dz, \qquad P_{i+1,i} = \int_{ \{|z| > R_{i} \} } K_{R_{i+1}}(z)\, dz.
\end{align*}
Keeping in mind that the integration domain in
\eqref{eq:sphere_interactions} depends on the relative ordering of the
indices, we group the interaction terms as follows
\begin{align*}
    \dot{R}_{i} - \dot{R}_{i+1} &= c_0 \Big[ \omega \Big( R_{i+1}^{-2s} - R_i^{-2s} \Big) + \Big( \sum_{j>i} P_{i,j} - \sum_{j>i+1} P_{i+1,j} \Big) + \Big( \sum_{j \le i} P_{i+1,j}- \sum_{j<i} P_{i,j} \Big) \Big] \\
    & = c_0 \Big[ \omega \Big( R_{i+1}^{-2s} - R_i^{-2s} \Big) + \Big( P_{i,i+1} - \sum_{j >i+1} [P_{i+1,j} - P_{i,j}]\Big) \\ &\qquad\quad +\Big( P_{i+1,i} - \sum_{j<i}[P_{i,j}-P_{i+1,j}] \Big) \Big].
\end{align*}
Since $R_{i+1}<R_i$, we have $\omega (R_{i+1}^{-2s}-R_i^{-2s}) >0$. Therefore, to prove \eqref{eq:non_collision}, it suffices to show that
\begin{equation}\label{eq:non_collision_inequalities}
    P_{i,i+1} - \sum_{j >i+1} [P_{i+1,j} - P_{i,j}] \geq 0, \qquad P_{i+1,i} - \sum_{j<i}[P_{i,j}-P_{i+1,j}] \geq 0.
\end{equation}
We will prove the first inequality. 
If $i=N-1$, the sum is empty,
and the claim follows immediately from $P_{i,i+1}\geq 0$. Otherwise,
for $p=0,\ldots,N-i-1$, define $\rho_p := R_{i+1+p}$. Thus, $\rho_0 = R_{i+1}$ and $\rho_{N-i-1} = R_N$. 
Since $\gamma>0$ is the smallest gap between consecutive radii, we have
\begin{align}\label{p_{p-1}-p_{p}bound}
    \rho_{p-1} - \rho_p \geq \gamma\qquad p=1,\ldots,N-i-1.
\end{align}
Next, define
\[
Q_p
:=
\int_{\{|z|<\rho_p\}}K_{R_i}(z)\,dz
=
P_{i,i+1+p},
\qquad
p=0,\ldots,N-i-1,
\]
and
\[
S_p
:=
\int_{\{|z|<\rho_p\}}K_{R_{i+1}}(z)\,dz
=
P_{i+1,i+1+p},
\qquad
p=1,\ldots,N-i-1.
\]
In this notation, the first inequality in
\eqref{eq:non_collision_inequalities} becomes
\[
Q_0-\sum_{p=1}^{N-i-1}\left(S_p-Q_p\right)\geq 0.
\]
In the definition of $Q_p$, perform the change of variables
$w=z-\gamma e$. Since $R_i-\gamma=R_{i+1}$, we have $|z-R_ie| = |w-R_{i+1}e|$.
Therefore, recalling the definition of $K_R$ in
\eqref{eq:spheres_kernel}, we obtain
\begin{align*}
    Q_p = \int_{ \{ |z|< \rho_p \} } K_{R_i}(z)\,dz = \int_{ \{ |w+\gamma e| < \rho_p \}} K_{R_{i+1}}(w)\, dw.
\end{align*}
On the other hand, \eqref{p_{p-1}-p_{p}bound} implies that   $\{ |w| < \rho_p \} \subseteq \{ |w +\gamma e| < \rho_{p-1} \}$. Therefore, 
\begin{align*}
    S_p =
\int_{\{|w|<\rho_p\}}K_{R_{i+1}}(w)\,dw
\leq
\int_{\{|w+\gamma e|<\rho_{p-1}\}}
K_{R_{i+1}}(w)\,dw
=
Q_{p-1}.
\end{align*}
It follows that
\begin{align*}
    \sum_{p=1}^{N-i-1}(S_p - Q_p) \le \sum_{p=1}^{N-i-1}(Q_{p-1} - Q_p) = Q_0 - Q_{N-i-1} = P_{i,i+1} - \int_{ \{ |y|<R_N \} } K_{R_i}\,dy \le P_{i,i+1}.
\end{align*}
 This proves the first inequality in
\eqref{eq:non_collision_inequalities}.
The second inequality follows
from an analogous argument.

From \eqref{eq:non_collision}, we deduce that
\[
\dot{\gamma}(t)
=
\dot{R}_i(t)-\dot{R}_{i+1}(t)
>0.
\]
Since $\dot{\gamma}>0$ at any point of differentiability of $\gamma$, and $\gamma$ is
locally Lipschitz continuous, it follows that $\gamma$ is
nondecreasing on $[0,T)$. This concludes the proof.
\end{proof}
By standard ODE theory, there exists a unique maximal solution to
\eqref{eq:sphere_ODE} with initial condition
\eqref{initialcond:spheres}, defined on an interval $[0,T)$. By
Lemma~\ref{lem:non_collision},
\[
\gamma(t)\geq\gamma(0)>0,
\qquad t\in[0,T),
\]
so no two adjacent spheres can collide. Consequently, the solution
remains in $\mathcal D$ as long as $R_N(t)>0$. Moreover,
\[
\dot R_N(t)
=
c_0\left(
-\frac{\omega}{R_N(t)^{2s}}
-\sum_{j<N}P_{N,j}(t)
\right)
\leq
-\frac{c_0\omega}{R_N(t)^{2s}},
\]
and hence the innermost sphere becomes extinct in finite time.
Therefore, the maximal existence time satisfies
\[
\lim_{t\to T^-}R_N(t)=0.
\]



\end{document}